\documentclass{amsart}

\usepackage[latin1]{inputenc}
\usepackage{amsfonts}
\usepackage{amsmath}
\usepackage{amsthm}
\usepackage{amssymb}
\usepackage{latexsym}
\usepackage{enumerate}

\newtheorem{theorem}{Theorem}[section]
\newtheorem{lemma}[theorem]{Lemma}
\newtheorem{proposition}[theorem]{Proposition}

\newtheorem{definition}[theorem]{Definition}

\numberwithin{equation}{section}

\newcommand{\cc}{\mathfrak{c}}
\newcommand{\N}{\mathbb{N}}
\newcommand{\BB}{\mathcal{B}}
\newcommand{\C}{\mathbb{C}}

\newcommand{\R}{\mathbb{R}}
\newcommand{\Z}{\mathcal{Z}}

\newcommand{\F}{\mathcal{F}}
\newcommand{\M}{\mathcal{M}}
\newcommand{\A}{\mathcal{A}}
\newcommand{\B}{\mathcal{B}}
\newcommand{\D}{\mathcal{D}}
\newcommand{\E}{\mathcal{E}}

\newcommand{\U}{\mathcal{U}}
\newcommand{\X}{\mathcal{X}}
\newcommand{\Y}{\mathcal{Y}}
\newcommand{\CC}{\mathcal{C}}

\newcommand{\bb}{\mathcal B(\ell_2)}

\newcommand{\K}{\mathcal{K}(\ell_2)}
\newcommand{\kk}{\mathcal{K}}
\newcommand{\QQ}{\mathcal{Q}(\ell_2)}

\begin{document}

\author{Piotr Koszmider}
\address{Institute of Mathematics of the Polish Academy of Sciences,
ul.  \'Sniadeckich 8,  00-656 Warszawa, Poland}
\email{\texttt{piotr.math@proton.me}}

\thanks{The  author was  partially supported by the NCN (National Science
Center, Poland) research grant no.\ 2020/37/B/ST1/02613.}

\thanks{The  author would like to thank Damian G{\l}odkowski and Clayton Suguio Hida for their feedback on 
 earlier versions
 of this   paper.}

\subjclass[2020]{46L05,  03E35,  54D30, }

\title[On masas of the Calkin algebra generated by projections]{On masas of the Calkin algebra generated by projections}

\begin{abstract}

Assuming the continuum hypothesis ({\sf CH}), we obtain complete $*$-isomorphic classification of
maximal abelian self-adjoint subalgebras (masas)  of the
  Calkin algebra $\QQ$ (bounded operators on a separable Hilbert space modulo compact operators)
generated by projections. In particular, for any compact totally disconnected Hausdorff space $K$
of weight not exceeding the continuum and not admitting $G_\delta$ points we construct under {\sf CH} a masa of $\QQ$
 which is $*$-isomorphic to the algebra $C(K)$
of complex-valued continuous functions on $K$.

This, among others, shows that masas of the Calkin algebra could have rather unexpected properties
compared to the previously known three $*$-isomorphic types of them generated by projections:
$\ell_\infty/c_0$, $L_\infty$ and $\ell_\infty/c_0\oplus L_\infty$.

It can be shown that some additional
set-theoretic hypothesis, like {\sf CH}, is necessary  for such results.
However, without making any additional set-theoretic assumptions we still construct a family
 of maximal possible cardinality (of the power set of $\R$) of pairwise non-$*$-isomorphic
 masas of $\QQ$ generated by projections and with  properties unlike the three above examples.

\end{abstract}

\maketitle

\section{Introduction}

For notation and terminology unexplained in the Introduction see Section 2.

This paper concerns maximal abelian self-adjoint subalgebras
(abbreviated as usual as {\sl masas}) of the Calkin algebra $\QQ=\bb/\K$ (\cite{calkin}), where
$\bb$ is the C*-algebra of all bounded linear operators on $\ell_2$ and $\K$ is its two-sided ideal of
compact operators (the quotient $*$-homomorphism from $\bb$ onto $\QQ$ will be denoted $\pi$). 
By the Gelfand-Naimark theorem every masa $\A$ of a unital C*-algebra, as a unital commutative C*-algebra, is $*$-isomorphic to
the C*-algebra $C(K)$ of all continuous complex-valued functions on a compact Hausdorff space $K$.
As usual $K$ will be called the {\sl Gelfand space} of $\A$. A  masa is generated by
projections (i.e. elements $P$ satisfying $P=P^2=P^*$)  if and only if its Gelfand space $K$ is totally disconnected, that is,
its clopen subsets form a basis of open sets of $K$.

All masas of $\bb$ have been characterized up to their positions in $\bb$ and
their $*$-isomorphism types long ago: $\A$ is a masa in $\B(\ell_2)$
if and only if there is a separable measure $\mu$ and a unitary $U:\ell_2\rightarrow L_2(\mu)$
such that $\A=\{U^*TU: T\in L_\infty(\mu)\}$, where $L_\infty(\mu)$ is considered as acting  on $L_2(\mu)$
by multiplication (C.6.11 \cite{ilijas-book}). In particular, all masas in $\bb$ are generated by projections.
There are only continuum many masas of $\bb$ and they fall into countably many unitary equivalence classes 
(Proposition 12.3.1 of \cite{ilijas-book}).

Masas of $\QQ$ are much more elusive.  A partial insight into their structure can be achieved through masas of $\bb$
due to a fundamental   result of
Johnson and Parrott (\cite{parrott}, 12.3.2 of \cite{ilijas-book}) which says that $\pi[\A]$ is a masa of $\QQ$ if $\A$ is a masa of $\bb$.
This allows us (see Lemma \ref{masa-masaQQ}) to characterize masas of $\QQ$ with commutative lifts to $\bb$ :
if $\A\subseteq \QQ$ is a masa with a commutative lift (i.e., such that there is
a commutative C*-subalgebra $\B$ of $\bb$ satisfying $\pi[\B]=\A$), then its
 Gelfand space 
is one of the following compact Hausdorff spaces:
\begin{enumerate}
\item $\N^*=\beta\N\setminus\N$, where $\beta\N$ is the \v Cech-Stone compactification of $\N$ with the discrete topology,
\item The Stone space $K_{Bor/\mathcal N}$ of the quotient Boolean algebra $Bor/\mathcal N$
of all Borel subsets of $[0,1]$ divided by
the ideal of Lebesgue measure zero sets,
\item The disjoint union of $\N^*$ and $K_{Bor/\mathcal N}$.
\end{enumerate}

The following results (as far as we know, the only in the literature concerning lifting masas of $\QQ$ or
their $*$-isomorphism types), 
although shedding little light on the nature of the
masas of $\QQ$, show that the above list (1) - (3) is far from complete:
\begin{enumerate}
\item[(i)]  There is a masa of $\QQ$ whose Gelfand space
$K$ is not totally disconnected (i.e., $C(K)$ is not generated by projections) and
so without commutative lift to $\bb$ (Brown, Douglas, Fillmore p. 126 of \cite{bdf}, cf. 12.4.3 of \cite{ilijas-book}).
\item[(ii)]  Assuming the continuum hypothesis ({\sf CH}) there is a masa in $\QQ$
which does not have a commutative lift to $\bb$, whose Gelfand space
$K$ is totally disconnected, i.e., it is generated by projections (Anderson \cite{pathology},
 the hypothesis of {\sf CH} was much weakened 
by Shelah and Steprans in \cite{ss} but it was shown in \cite{dilip} that this hypothesis is still beyond the axioms {\sf ZFC}).
\item[(iii)] There are $2^\cc$ masas in $\QQ$ 
which do not have commutative lifts to $\bb$, where $\cc$ denotes the cardinality 
of the continuum (Akemann and Weaver;  \cite{akemann-pure}). We do not know if these masas are generated by
projections or not.

\end{enumerate}

The topic of obstructions to liftings of commutative subalgebras in $\QQ$ generated by
 projections in general (not necessarily of masas of $\QQ$)
has been investigated separately in papers like   \cite{ilijas-wofsey} (see 5.35), \cite{tristan} or \cite{vaccaro-pacific},
cf. Section 14.3 of \cite{ilijas-book}.

Our contribution to understanding masas of $\QQ$ presented in this paper is:

 \begin{theorem}\label{main}
 $ $
 \begin{enumerate}
 \item[(A)] There is a family  of $2^\cc$ pairwise non-homeomorphic totally disconnected compact
 Hausdorff spaces $K$ such that $C(K)$ is $*$-isomorphic to a masa in $\QQ$ without a commutative lift (no additional
 set-theoretic hypothesis needed). These $K$s are not F-spaces ($C(K)$s are not SAW*-algebras)\footnote{
 A compact Hausdorff space $K$ is called an F-space  if the closures of disjoint open 
 $F_\sigma$ subsets  of $K$ are disjoint. A C*-algebra
 is called SAW* if given its two orthogonal positive elements $a, b$ there is a positive $c$ such that $ac=a$ and $bc=0$
 (see \cite{saw}).} and the masas $C(K)$s do not admit conditional expectations from $\QQ$.
 \item[(B)] Assuming  the continuum hypothesis,  for a 
 compact Hausdorff totally disconnected space $K$
the C*-algebra $C(K)$ is $*$-isomorphic to a masa of $\QQ$ generated by projections
 which does not have a commutative lift if and only if  the topological weight of $K$ does not
 exceed $\cc$ and $K$ does not admit any $G_\delta$ point.
 Moreover, there are more than $\cc$ pairwise unitarily non-equivalent such masas of $\QQ$
 $*$-isomorphic to each $C(K)$ as above.
 \end{enumerate}
 \end{theorem}
 \begin{proof}
 For (A) use Propositions \ref{a-prop},  \ref{two-to-c} and Lemma \ref{restriction-M}. For (B) use Propositions \ref{ch-proposition}
 and \ref{gdelta}.\end{proof} 
 So item (B) solves (under {\sf CH}) the problem of classifying masas of the 
 Calkin algebra mentioned already in the Introduction to the paper \cite{pathology} of Anderson.
 
 In fact, the condition that $K$ does not admit any $G_\delta$ point is necessary for $C(K)$
 to be a masa of $\QQ$ without any additional set-theoretic assumption (Proposition \ref{gdelta}).
 For the sufficiency of this condition 
 some additional hypothesis like {\sf CH} is necessary as it is known that consistently
 some  $C(K)$s for $K$s without $G_\delta$-points, do not embed into the Calkin algebra at all, 
 so in particular cannot be masas of $\QQ$.
 These are, for example:
 \begin{enumerate}
\item[(a)] the C*-algebra generated by the Boolean algebra of projections isomorphic to
 the Boolean algebra of all subsets of $\N$ modulo sets of asymptotic density zero
 or in general $\wp(\N)$ divided by a meager dense ideal (Vignati; Corollary 5.3.14 \cite{vignati-thesis} 
 cf. McKenney, Vignati; Theorem 9.1 of \cite{vignati-mckenney}).
 \item[(b)]  $C(\{0,1\}^{\omega_1}\times [0,\cc])$ or in general any  C*-algebra 
containing the C*-algebra $C([0, \cc])$. Here $[0,\cc]$ is considered with the order topology (Vaccaro; \cite{vaccaro-thesis}).
  \item[(c)] $C(\beta(\N\times\N^*))$  or in general any C*-algebra containing the C*-algebra $\ell_\infty(c_0(\cc))$
(G{\l}odkowski, Koszmider; \cite{damian}). 
 \end{enumerate}
 Also the mere class of compact  spaces without $G_\delta$-points is sensitive
 to extra set theoretic assumptions (e.g. under the Proper Forcing Axiom such spaces cannot be countably tight,
 see Theorem 6.3. (iii) of \cite{elementary})
However, assuming {\sf CH} any commutative C*-algebra of density $\cc$  $*$-embeds
 into the image in $\QQ$ of the atomic masa\footnote{However, assuming the Open Coloring Axiom
 we witness the situation when  $L_\infty$ does not embed into the  image   of the atomic masa
 in $\QQ$ by \cite{measure} but of course $L_\infty$ does embed into $\QQ$.} (see
\cite{handbook-van-mill}),
in fact,  then even any C*-algebra of density $\cc$  can be $*$-embedded into $\QQ$  by \cite{universal-calkin}. 

Both items (A) and (B) of Theorem \ref{main} provide many new, compared
 to examples (1) - (3), $*$-isomorphic types of masas in $\QQ$
generated by projections:
When we may employ {\sf CH} 
we obtain maximal possible variety of masas of $\QQ$ generated by projections. By (a) - (c) such a variety cannot be obtained
without an additional set-theoretic hypothesis. By item (A) of Theorem \ref{main}
without an additional set-theoretic hypothesis
we have a family of diverse $*$-isomorphic types of masas of $\QQ$ generated by
projections of maximal possible cardinality.  There are still many fundamental questions concerning masas of $\QQ$
left open, for example: can a version of item (B) of Theorem \ref{main} be proved for masas not generated by projections?
what are exactly the masas of $\QQ$ which can be constructed without any additional set-theoretic hypothesis? 
 
 The approach to item (A) of Theorem \ref{main} is to view $\bb$ as 
 $\bb\otimes M_2$ where $M_2$ stands for the C*-algebra of 
 two by two complex matrices and to construct all the $2^\cc$ C*-algebras of item (A)
of Theorem \ref{main} as intermediate algebras in between
$\pi[\D(\E)\otimes \C I_2]$ and  $\pi[ \D(\E)\otimes M_2]$, where $\D(\E)$ is the algebra of all operators diagonal
with respect to some fixed orthonormal basis $\E$ of $\ell_2$ and $I_2$ is the two by two identity matrix.  
It is another recent case, besides \cite{pure}, showing that considering sequences of finite dimensional objects  in $\bb$ leads to
interesting discoveries in the Calkin algebra.

To add ``a twist'' which prevents the masa from
having a lift we may use a maximal  infinite almost disjoint family of infinite subsets of $\N$ (i.e., 
such family that $A\cap B$ is finite for any two distinct members of $A, B$ of the family) and choose
a different lift on each subspace corresponding  to each element of the family.
This was also the approach of \cite{ss}, but we do it differently which allows us to work 
without any additional set-theoretic assumption. To obtain $2^\cc$ C*-algebras of item (A)
we consider more general families of subsets of $\N$ instead of almost disjoint families
which we call wide families (Definition \ref{def-wide}).  A clear description of the Gelfand space
of one masa from item (A) of Theorem \ref{main} is possible in the case where the
wide family is maximal almost disjoint (Proposition \ref{few-projections} (3)  and Definition \ref{def-BX}).

In fact, item (B) of Theorem \ref{main} seems more striking as it shows
that nothing (in {\sf ZFC}) prevents  masas of $\QQ$ from having rather diverse properties, including
properties quite far from the liftable examples whose Gelfand spaces are listed  in (1) - (3).
So, in particular masas in $\QQ$ may be (under {\sf CH}) of the following form:
\begin{itemize}
\item $C(K)$ for $K$ being a compact group $\{0,1\}^{\cc}$,
 which in particular has many nontrivial convergent sequences, so these masas do not
 have the Grothendieck\footnote{The Banach space $\bb$ has the Grothendieck property by \cite{pfitzner} and 
 the property is preserved by Banach quotients (Proposition 3.1.4 of \cite{kania}), also
 Banach spaces $C(K)$ lack this property if $K$ admits a nontrivial convergent sequence (Chapter 4 of \cite{kania}).}
  or the Nikodym property (cf. \cite{darst}) and so do not admit any conditional expectations from
 $\QQ$.
\item Commutative von Neumann  algebra  $L_\infty(\{0,1\}^\cc)$ (different than separable $L_\infty$),
\item $2^\cc$ pairwise non-$*$-isomorphic C*-algebras generated by 
$2^\cc$ nonisomorphic complete nonatomic Boolean algebras (e.g. from \cite{monk}); each of them admitting a
conditional expectation by Tomiyama theorem (Theorem 3.3.4 \cite{ilijas-book}) and by
the fact that they are $1$-injective Banach spaces.
\item Exotic examples of $C(K)$s like, for example, for $K$ being Efimov spaces (see e.g. \cite{efimov}) or
$C(K)$s having few linear operators (\cite{few}).
\end{itemize}
It seems that the consistency  result like in items (a) - (c)  does not refer to the bulleted
examples above and so we do not know if it is provable without additional set-theoretic hypothesis that
there are  masas of $\QQ$ with the above properties. 
 Definitely, one expects that
axioms like, for example, the Open Coloring Axiom, provide a better behaved theory of the Calkin algebra
(compared to the theory under {\sf CH}) as already indicated, for example, by the results of \cite{farah-ann}

The main difference of our argument in the proof of Theorem \ref{main} (B)
compared to the one of Anderson in \cite{pathology} is that we are
able to control the $*$-isomorphism type of the masa we construct by transfinite induction of length $\omega_1$.
This is due to replacing the Voiculescu Bicommutant Theorem (Corollary 1.9 in \cite{voiculescu}) in the argument
of Proposition \ref{ch-proposition} by our Lemma \ref{main-lemma} which  refers only to the 
commutative situation (unlike \cite{voiculescu})
but  is more sensitive to the $*$-isomorphism of the constructed masa.

The paper is organized as follows:
in the next Section 2 we review terminology, notation and state or sometimes prove some preliminary or
folklore results. 
Section 3 is devoted to proving item (B) of Theorem \ref{main}
and Section 4 is devoted to proving item (A) of Theorem \ref{main}.

\section{Preliminaries}

\subsection{Notation}

$1_X$ denotes the characteristic function of a set $X$. The restriction of a function $f$
to a set $X$  is denoted by $f|X$. $\N$ stands for the set of all nonnegative integers. $X^Y$
denotes the set of all functions from $Y$ into $X$.  If $\sigma\in \{0,1\}^n$ for some $n\in \N$
by $[\sigma]$ we mean the clopen subset of $\{0,1\}^\N$ defined by
$$[\sigma]=\{x\in \{0,1\}^\N: x|n=\sigma\}.$$   We write $A\subseteq^*B$ when 
$A\setminus B$ is finite and $A=^*B$ if $A\subseteq^*B$ and $B\subseteq^* A$.
If $(a_n)_{n\in \N}$ is a sequence of complex numbers and $X\subseteq \N$ is infinite, we say 
that $\lim_{n\in X}a_n=a$ if $\lim_{k\rightarrow \infty }a_{n_k}=a$,
where $(n_k)_{k\in \N}$ is the increasing enumeration of elements of $X$.
For $X\subseteq \N$ by $c_0(X)$ we mean the set of 
all sequences converging to $0$ and equal $0$ on $n\in \N\setminus X$; 
$c_0=c_0(\N)$.

\subsection{Boolean algebras and compact spaces}

Boolean algebra terminology is standard and should follow e.g., the book \cite{koppelberg}.
We use   $\wedge, \vee, -$ for Boolean operations and $1, 0$ for the unit and zero of a Boolean algebra.
The Boolean order is defined as $A\leq B$ if $A\wedge B=A$. 
Filters are subsets of Boolean algebras which do not contain $0$
and  which are upward closed in the Boolean order and closed under $\wedge$.
Ultrafilters are maximal filters. If $\U$ is an ultrafilter of a Boolean algebra $\A$, then
$A\in \U$ or $1-A\in \U$ for every $A\in \A$.  

We will deal with homomorphisms of Boolean algebras
and in particular Boolean monomorphisms (homomorphisms whose kernels are $\{0\}$) and
Boolean isomorphisms.  We  use the adjective Boolean before
the noun homomorphism, monomorphism etc, to distinguish it from 
$*$-homomorphisms concerning C*-algebras.

For every Boolean algebra
$\A$ there is a compact Hausdorff totally disconnected (with open basis consisting of clopen sets) 
space $K_\A$ such that there exists a
Boolean isomorphism $h: \A\rightarrow Clop(K_\A)$, where $Clop(K)$ denotes the Boolean algebra
of clopen subsets of $K$ with $\cap, \cup$ and the complement as the Boolean operations. 
$K_\A$ is called {\sl the Stone space} of $\A$. Its points correspond to
ultrafilters of $\A$. Namely, by the compactness of $K_\A$, every ultrafilter in $\A$ is of the form
$\U_x=\{h^{-1}(U): x\in U, U\in Clop(K_\A)\}$ for some $x\in K_\A$. 
In fact,  the Stone duality says that the category of compact Hausdorff totally disconnected spaces with continuous maps
is dual to the category of Boolean algebras with homomorphism
(Chapter 3 of \cite{koppelberg}, Chapter 16.3 of \cite{semadeni}).

\begin{definition}\label{def-ba-split}
Suppose that $\A\subseteq \B$ are Boolean algebras, $B\in\B\setminus\A$
and $\U\subseteq \A$ is an ultrafilter in $\A$. We say that $B$ splits $\U$
if $A\wedge B\not =0\not= A\wedge (1-B)$ whenever $A\in \U$.
\end{definition}

A point in a topological space is called a $G_\delta$ point if its  singleton  is the intersection of a countable family of open sets.
For a point in a compact Hausdorff space being $G_\delta$ is equivalent to having a countable neighborhood basis.

\begin{lemma}\label{gdelta-split} Suppose that $K$ is totally disconnected  compact Hausdorff space which does
not admit any $G_\delta$ point and $\A$ is a countable Boolean subalgebra of $Clop(K)$.
For every ultrafilter $\U\subseteq \A$ there is $V\in Clop(K)\setminus \A$ which splits $\U$.
\end{lemma}
\begin{proof} Let $\U\subseteq \A$ be an ultrafilter of $\A$. Take any ultrafilter $\U'$ of $Clop(K)$
which extends $\U$. By the compactness of $K$ the ultrafilter $\U'$ is of the form $\{U\in Clop(K): x\in U\}$
for some $x\in K$. As $K$ does not admit any $G_\delta$ point, $\U$ does not form a neighborhood basis for
$x$. So there is $V\in \U'$ such that   no $U\in \U$ satisfies $U\subseteq V$, so every $U\in \U$
satisfies $U\cap(1-V)\not=\emptyset$.
Also $U\cap V\not=\emptyset$ because $x\in U, V$.
\end{proof}

Two concrete Boolean algebras that we will be dealing with are the Boolean algebra $\wp(\N)$ 
of all subsets of nonnegative integers $\N$ and its quotient $\wp(\N)/Fin$ by 
 the ideal $Fin$ of finite subsets of $\N$. 
$\rho$ stands for the quotient Boolean epimorphism $\rho:\wp(\N)\rightarrow \wp(\N)/Fin$.
The Stone space of $\wp(\N)$ is canonically homeomorphic to the \v Cech-Stone
compactification $\beta\N$ of $\N$ and the Stone space of $\wp(\N)/Fin$
 is canonically homeomorphic with the \v Cech-Stone
remainder $\beta\N\setminus\N$ denoted by $\N^*$.

\begin{lemma}\label{injective} Suppose that $\A, \B$ are countable Boolean algebras with $\A\subseteq \B$.
Let $h: \A\rightarrow \wp(\N)/Fin$ be a Boolean monomorphism. Then there is
an extension $h'$ of $h$ to a Boolean monomorphism  $h':\B\rightarrow \wp(\N)/Fin$.
\end{lemma}
\begin{proof} This is a folklore result which is the basic ingredient of any proof of Parovi\v cenko theorem
(e.g. \cite{koppelberg}, \cite{handbook-van-mill}), however we did not find it as an isolated lemma in the literature.
The quickest way of concluding it from the explicitly published result seems to be to use its topological version, proved
for example in Lemma (iii) on page 48 of  \cite{engelking} and apply the Stone duality.
\end{proof}

\begin{lemma}\label{boolean-lift} Suppose that $\A$ is a countable Boolean subalgebra of
$\wp(\N)/Fin$. Then there is a Boolean subalgebra $\A'$ of $\wp(\N)$ such that
$\rho|\A': \A'\rightarrow \A$ is a Boolean isomorphism.
\end{lemma}
\begin{proof} It is well-known that every countable Boolean algebra is a subalgebra
of a countable free Boolean algebra i.e.,  $Clop(\{0,1\}^\N)$ because the Cantor set continuously maps onto any metric compact Hausdorff space. 
So by Lemma \ref{injective} we may assume that
$\A$ is free, that is generated by a countably infinite independent set of generators $A_n$ for $n\in\N$.
Note that any set $\{A_n': n\in \N\}\subseteq  \wp(\N)$  must be independent
if $\rho(A_n')=A_n$ for every $n\in \N$. 
So the algebra $\A'\subseteq \wp(\N)$ generated by such $\{A_n': n\in \N\}$
is the desired $\A'$ by the freeness of both algebras.

\end{proof}

\subsection{C*-algebras}
Our terminology concerning C*-algebras should follow \cite{blackadar}, \cite{murphy} and \cite{ilijas-book}.
Our models of $n$-th dimensional Hilbert spaces are $\ell_2(n)$ and the model of
an infinite dimensional  separable Hilbert space is $\ell_2$.  If $F\subseteq\N$ by $\ell_2(F)$ we denote
the subspace of $\ell_2$ of $F$-supported vectors.
For a  closed subspace $V$ of a Hilbert space $V^\perp$ denotes
its orthogonal complement.
The commutator of $T, S$ from
a C*-algebra is denoted by $[T, S]=TS-ST$. The C*-algebra of all
bounded linear operators on $\ell_2$ ($\ell_2(n)$) is denoted  $\bb$ ($\mathcal B(\ell_2(n))$). 
The ideal of compact operators on $\ell_2$ is denoted $\K$. 
The Calkin algebra $\bb/\K$ is denoted $\QQ$. The quotient
map from $\bb$ onto $\QQ$ is denoted $\pi$. We will often be using 
the following two facts about compact operators:
\begin{itemize}
\item If $A\in \K$ and $(v_n)_{n\in \N}$ is a sequence of unit, pairwise orthogonal vectors in $\ell_2$,
then $\lim_{n\in \N}\|A(v_n)\|=0$.
\item If  $(v_n)_{n\in \N}$ is an orthonormal basis of $\ell_2$ and $A\in \bb$ satisfies
$\sum_{n\in \N}\|A(v_n)\|<\infty$, then $A\in \K$.
\end{itemize}
Both of them follow from the fact that an operator in $\bb$ is compact if 
and only if it can be norm approximated by finite dimensional operators. 

If $\F\subseteq\QQ$, then 
$C^*(\F)$  denotes the C*-subalgebra of $\QQ$ generated by $\F$.
All C*-algebras we deal with will be unital subalgebras of $\bb$ or $\QQ$ or algebras
of the form $C(K)$ of complex-valued continuous functions on a compact Hausdorff space $K$.
By the Gelfand-Naimark theorem all commutative C*-algebras are $*$-isomorphic 
to some member of the latter class of C*-algebras.
Recall that a $*$-isomorphism of C*-algebras is an isometry (1.3.3 of \cite{ilijas-book}).
An element $S$ of a C*-algebra is self-adjoint (a projection) if $S=S^*$ ($S=S^2=S^*$).
Recall that if $S\in \QQ$ is self-adjoint (a projection), then there is $T\in \bb$ such that $\pi(T)=S$
and $T$ is self-adjoint ($T$ is a projection) (Lemma 2.5.4 and 3.1.13 of \cite{ilijas-book}).
The spectrum of an element $T$ of a unital C*-algebra is denoted by $\sigma(T)$.
Recall that it does not grow when we pass to a quotient,
in particular $\sigma(\pi(T))\subseteq \sigma(T)$ for $T\in \bb$  (Lemma 2.5.7 in \cite{ilijas-book}).

If $\E=(e_n)_{n\in \N}$ is an orthonormal basis of $\ell_2$, then 
$\D(\E)$ denotes  the commutative C*-algebra of all operators diagonal with respect to $\E$, i.e.,
such operators that $e_n$ is its eigenvector  for every $n\in \N$. 
Weyl-von Neumann theorem asserts that for every self-adjoint $T\in \bb$
there is an orthonormal basis $\E$ and  $S\in \K$ such that $T+S\in \D(\E)$ (V. 4.1.2 (ii) of \cite{blackadar}).
In fact any countable set of commuting projections can be simultaneously diagonalized
modulo a compact operator (Lemma 12.4.4 of \cite{ilijas-book}). We will need a more specific version of this
in Lemma \ref{calkin-lift}.

\subsection{Boolean subalgebras of $\QQ$ and their liftings}

The set of projections  of a $C^*$ algebra $\A$ will be denoted
by $Proj(\A)$. All projections in C*-algebras of the form $C(K)$ for $K$ compact Hausdorff
are  characteristic functions $1_U$ for clopen $U\subseteq K$. The Boolean operations
in the Boolean algebra $Clop(K)$ can be translated to C*-algebraic operations on
the functions $1_U$ for $U\in Clop(K)$ as $1_{U\cap V}=1_U1_V$, $1_{U\cup V}=1_U+1_V-1_{U}1_V$
and $1_{K\setminus U}=1_K-1_U$. By the Gelfand-Naimark theorem this means that 
families of commuting projections in C*-algebras generate Boolean algebras with the
 above operations (see e.g. \cite{bade}, \S 30 of \cite{halmos}):

\begin{definition}\label{ba-projections} If $P, Q$ are commuting projections in a C*-algebra $\A$, then we define
\begin{itemize}
\item $P\wedge Q=PQ$,
\item $-P=I-P$,
\item $P\vee Q=P+Q-PQ$.
\end{itemize}
\end{definition}

Whenever we say that $h:\A\rightarrow \bb$ or $h:\A\rightarrow \QQ$  is a Boolean
homomorphism (monomorphism) we mean that the range of $h$ consists of commuting
projections and it is a Boolean algebra with the operations as in Definition  \ref{ba-projections} and
such that $h(1)$ is the unit of $\bb$ or $\QQ$ and $h(0)$ is the zero of $\bb$ or $\QQ$ respectively.

\begin{lemma}\label{cstar-ba-iso} Suppose that the C*-algebra $C^*(\A)$ is generated by a Boolean algebra
of  commuting projections $\A\subseteq\QQ.$  Then $C(K_\A)$ and $C^*(\A)$ are $*$-isomorphic, where $K_\A$
is the Stone space of $\A$.
\end{lemma}
\begin{proof} As $C^*(\A)$ is generated by commuting self-adjoint elements, it  is commutative,
and so there is a $*$-isomorphism $\phi: C(L)\rightarrow C^*(\A)$ for some compact Hausdorff $L$.
Points of $L$ must be separated by the generating set  $\phi^{-1}[\A]$, which consists of projections, so
$\{\phi^{-1}(A): A\in \A\}$ consists of  characteristic functions of clopen sets  from a Boolean algebra
which separates points of $L$. In particular $L$ contains no connected subspace which
has more than one point, and so clopen sets whose characteristic functions are in $\phi^{-1}[\A]$
form an open base for $L$. So $\phi^{-1}[\A]$ must be the entire set of
characteristic functions of clopen subsets of $L$ (Lemma 7.6 of \cite{koppelberg})

So the Boolean algebras $Clop(K_\A)$ and $Clop(L)$ are isomorphic, both spaces are totally
disconnected and compact so   $K_\A$ and $L$ are homeomorphic (Theorem 7.10 of \cite{koppelberg})
and so $C(K_\A)$ and $C^*(\A)$ are $*$-isomorphic.
\end{proof}

\begin{lemma}\label{proj} Suppose that $\E=(e_n)_{n\in \N}$ is an orthonormal basis for $\ell_2$.
$Proj(\pi[\D(\E)])=\pi[Proj(\D(\E))]$.
\end{lemma}
\begin{proof} The inclusion $Proj(\pi[\D(\E)])\supseteq \pi[Proj(\D(\E))]$ is trivial.
For the other inclusion fix $Q\in Proj(\pi[\D(\E)])$ and $A\in \D(\E)$ such that
$\pi(A)=Q$. Let $(a_n)_{n\in \N}$ be the diagonal entries of the matrix of $A$.
As $\pi(A)=\pi(A^2)$, we have $\lim_{n\in \N}(a_n-a_n^2)=0$.
This means that $0$ and $1$ are the only accumulation points of $(a_n)_{n\in \N}$, so
putting $X=\{n: |a_n-1|<1/2\}$ gives us that $\pi(R)=\pi(A)$, where $R\in \D(\E)$
has values $(1_X(n))_{n\in \N}$ on the diagonal of its matrix. Clearly $R$ is a projection in $\D(\E)$.
So $\pi(A)=\pi(R)\in \pi[Proj(\D(\E))]$.
\end{proof}

\begin{lemma}\label{pi} Suppose that $\E=(e_n)_{n\in \N}$ is an orthonormal basis for $\ell_2$.
Then $h:\wp(\N)\rightarrow Proj(\D(\E))$ sending $A\subseteq\N$ to the projection onto $\overline{span}\{e_n: n\in A\}$
is a Boolean isomorphism. Moreover there is a Boolean isomorphism 
$h':\wp(\N)/Fin\rightarrow Proj(\pi[\D(\E)])$ such that $h'\circ\rho=\pi\circ h$.
\end{lemma}
\begin{proof} 
Consider $\{0,1\}^\N$ with operations $1_A\wedge 1_B=1_A1_B=1_{A\cap B}$, 
$1_A\vee 1_B=1_A+1_B-1_A1_B=1_{A\cup B}$, $-1_A=1-1_A=1_{\N\setminus A}$. It is clear that
 $h_0: \wp(\N)\rightarrow \{0,1\}^\N$  given by
$h_0(A)=1_A$ is a Boolean isomorphism. Let $h_1: \{0,1\}^\N\rightarrow Proj(\D(\E))$
send $1_A$ to the operator whose matrix has $1_A$ on the diagonal. 
It is clear that $h_1$ is a Boolean isomorphism. As $h=h_1\circ h_0$, we conclude
the first part of the lemma.  To prove the second part
we need to note that $h[Fin]=\K\cap Proj(\D(\E))$, which
is clear since a projection is compact if and only if it is finite-dimensional.
\end{proof}

\begin{lemma}\label{calkin-lift}
Suppose that $\A\subseteq \QQ$ is a countable Boolean algebra of projections. Then there 
is an orthonormal basis $\E$ of $\ell_2$ such that
\begin{enumerate}
\item $\A\subseteq \pi[Proj(\D(\E))]$.
\item  There is a Boolean subalgebra $\B$ of $\D(\E)$ such that
$\pi|\B: \B\rightarrow \A$ is a Boolean isomorphism.
\end{enumerate}
\end{lemma}

\begin{proof} Let $C^*(\A)\subseteq \QQ$ be the C*-subalgebra of $\QQ$ generated by $\A$.
It is commutative as $\A$ consists of commuting self-adjoint elements, and it has real rank zero since
$C^*(\A)$ is $*$-isomorphic to $C(K_\A)$ by Lemma \ref{cstar-ba-iso}, where $K_\A$ is totally disconnected.
So by Lemma 12.4.4 (1) of  \cite{ilijas-book} which says that commutative real rank zero
C*-subalgebras of $\QQ$ have diagonal lifts  there is an orthonormal basis $\E$ of $\ell_2$
and  a C*-algebra $\B'\subseteq \D(\E)$ such that $\pi[\B']=C^*(\A)$.
By Lemma \ref{proj} we obtain (1)

To prove item (2) note that by the second part of Lemma \ref{pi} the Boolean
algebra $\pi[Proj(\D(\E))]$ is an isomorphic copy
of  $\wp(\N)/Fin$, so we are in the situation of 
 Lemma \ref{boolean-lift} and  applying this lemma we obtain the required 
Boolean subalgebra $\B$ of $\D(\E)$.

\end{proof}

\begin{lemma}\label{injective-Q}  Suppose that $\A, \B$ are countable Boolean algebras with $\A\subseteq \B$.
Let $h: \A\rightarrow Proj(\QQ)$ be a Boolean monomorphism. Then there is
an extension $h'$ of $h$ to a Boolean monomorphism  $h':\B\rightarrow Proj(\QQ)$.
\end{lemma}
\begin{proof}
By Lemma \ref{calkin-lift} there is an orthonormal basis $\E$ of $\ell_2$ and a Boolean 
subalgebra $\A'$ of $Proj(\D(\E))$ such that $\pi|\A':\A'\rightarrow h[\A]$ is a Boolean isomorphism.
By Lemma \ref{pi} the Boolean algebra $\pi[Proj(\D(\E))]$ is isomorphic to $\wp(\N)/Fin$.
So $h[\A]$ can be considered as a subalgebra of this isomorphic copy of $\wp(\N)/Fin$ inside $\QQ$.
So we  may use Lemma \ref{injective} to extend $h$ to a Boolean monomorphism 
$h':\B\rightarrow \pi[Proj(\D(\E))]\subseteq \QQ$.
\end{proof}

\begin{definition}\label{def-split}Suppose that $\A\subseteq \QQ$ is a Boolean algebra of projections
and $T\in \bb$ is self-adjoint such that $\pi(T)$ commutes with all elements of
$\A$. Let $\U$ be an ultrafilter of $\A$. We say that $T$ C*-splits $\U$ if there are
two distinct multiplicative states  $x_1, x_2$ of $C^*(\A\cup\pi(T))$  such that
$$\{A\in \A: x_1(A)=1\}=\{A\in \A: x_2(A)=1\}=\U.$$
\end{definition}

\begin{lemma}\label{split-point} Suppose that $\A\subseteq \QQ$ is a Boolean algebra of projections
and $T\in \bb$ is self-adjoint such that $\pi(T)$ commutes with all elements of
$\A$ but $\pi(T)\not\in C^*(\A)$. Then there is an ultrafilter $\U$ 
 of $\A$ which is $C^*$-split by $T$.
\end{lemma}
\begin{proof} As 
$\pi(T)$ commutes with all elements of
$\A$, it commutes with the entire $C^*(\A)$. Since $T$ is self-adjoint, this means that
$C^*(\A\cup\pi(T))$ is commutative and hence 
there is  a $*$-isomorphism $\phi: C^*(\A\cup\pi(T))\rightarrow C(L)$ for some   compact Hausdorff $L$.
Since $\pi(T)\in  C^*(\A\cup\pi(T))\setminus C^*(\A)$ this means, 
by the complex Stone-Weierstrass theorem, that $\phi[C^*(\A)]$ cannot
separate points of $L$, i.e., there are $y_1, y_2\in L$
 which are not separated by elements of  $\phi[C^*(\A)]$.
Put $x_i=\delta_{y_i}\circ \phi$ for $i=1,2$.   The non-separation by
$C^*(\A)$ means that $x_1(A)=\delta_{y_1}(\phi(A))=\delta_{y_2}(\phi(A))=x_2(A)$ for all $A\in \A$, so
put $x=x_1|C^*(\A)=x_2|C^*(\A)$. As $x_i$s are multiplicative states, so
is $x$. Let $\U=\{A\in \A: x(A)=1\}$. Note that
as $x(AB)=x(A)x(B)$ holds for every $A, B\in \A$, the set $\U$ is a filter.
Also $1=x(1)=x(A\vee(1-A))=x(A)+x(1-A)$.  As $A, 1-A$ are projections
this implies that $A\in \U$ or $1-A\in \U$ for every $A\in \A$, hence $\U$ is an ultrafilter, as needed.
\end{proof}

\begin{lemma}\label{c-split}Suppose that $\A\subseteq \QQ$ is a Boolean algebra of projections
and $T\in \bb$ is self-adjoint such that $\pi(T)$ commutes with all elements of
$\A$. Suppose that  $\U$ is an ultrafilter of $\A$ which is C*-split by $T$.
Then there are reals $c_1<c_2$ such that
$$
c_1,c_2\in\sigma_{A\QQ A}(A\pi(T)A)
$$
for every $A\in\U$, where the spectrum is computed in the corner
$A\QQ A$, whose unit is $A$.
\end{lemma}
\begin{proof} By Definition \ref{def-split}
there are
two distinct multiplicative states $x_1, x_2$ of $C^*(\A\cup\pi(T))$  such that
$$\{A\in \A: x_1(A)=1\}=\{A\in \A: x_2(A)=1\}=\U.$$
Since $\A=\{A\in \A: A\in \U\  \hbox{or}\ 1-A\in \U\}$, we have $x_1|\A=x_2|\A$
which means that 
$$c_1=x_1(\pi(T))\not=x_2(\pi(T))=c_2\leqno (*)$$ 
as $x_1$ and $x_2$ are distinct. The numbers $c_i$ are real as $\pi(T)$ is self-adjoint.
$x_i(A\pi(T)A)=c_i$ for every $A\in \U$ by Definition \ref{def-split} and
since $x_i$s are multiplicative for $i=1,2$. 

The invertibility of $A\pi(T)A-c_iA$ in ${A C^*(\A\cup\{\pi(T)\}) A}$ would yield
$A=(A\pi(T)A-c_iA)S$ for some
$S\in A C^*(\A\cup\{\pi(T)\})A$ as $A$ is its unit, and hence
$$1=x_i(A)=x_i(A\pi(T)A-c_iA)x_i(S)=0,$$
which leads to contradiction. So $$c_i\in \sigma_{A C^*(\A\cup\{\pi(T)\})A}(A\pi(T)A)$$ for $i=1,2$.
Since the spectrum does not change when we pass to a bigger unital C*-algebra,
with the same unit (Ex. 1.11.6 of \cite{ilijas-book}), the
conclusion follows after renaming $c_1,c_2$, if necessary, so that
$c_1<c_2$. \end{proof}

\subsection{Masas in $\bb$ and in $\QQ$}

A masa in a C*-algebra $\A$ is a maximal (with respect to inclusion) commutative C*-subalgebra of $\A$.
We will often use a theorem of Johnson and Parrott (see \cite{parrott}) which says that
if $\A\subseteq\bb$ is a masa of $\bb$, then $\pi[\A]$ is a masa of $\QQ$. 
While constructing masas of $\QQ$ the following lemma will be very useful:

\begin{lemma}[p. 72 of \cite{ss}]\label{sa-masa} If $\A$ is a commutative C*-algebra of $\QQ$ then
$\A$ is a masa of $\QQ$ if and only if for every self-adjoint $S\in \QQ\setminus \A$ there is 
$R\in \A$ which does not commute with $S$.
\end{lemma}
\begin{proof}
For the forward implication suppose that  there is a self-adjoint $S\in \QQ\setminus \A$ which  
commutes with all $A\in \A$. As $S=S^*$ we have that $C^*(\A\cup\{S\})$ is commutative as well, 
which shows that $\A$ is not a masa of $\QQ$.

For the reverse implication 
suppose that $\A$ is  a commutative C*-subalgebra of $\QQ$  which is not a masa  and we will
find a self-adjoint element of $ \QQ\setminus \A$ which  
commutes with the entire $\A$. Let $\A\subseteq \B\subseteq \QQ$ be a masa of $\QQ$ such that there
is $T\in \B\setminus\A$.  Then $T^*\in \B$ and both $T, T^*$ commute with
all elements of $\A$ and $TT^*=T^*T$ 
as they are in a bigger commutative algebra $\B\supseteq\A$. So self-adjoint $S=(T+T^*)/2\in \B$
and  $S'=(T-T^*)/2i\in \B$ commute with all elements of $\A$ as well.
If $S, S'\in \A$, then so is $T=S+iS'$, which is not the case, so one of $S, S'\not\in \A$
is the required self-adjoint element.
\end{proof}

\begin{lemma}\label{masa-masaQQ} Suppose that $\A\subseteq\QQ$ is a masa of $\QQ$ which has a commutative lift.
Then $\A$ is the image under $\pi$ of a masa of $\bb$.  
\end{lemma}
\begin{proof}
Let $\B\subseteq \bb$ be a  commutative C*-algebra such that $\pi[\B]=\A$.
Then there is a masa $\B'\subseteq \bb$ such that $\B\subseteq \B'$.
By the Johnson and Parrott theorem  $\pi[\B']\supseteq\A$ is a masa of $\QQ$, so
$\pi[\B']=\A$.
\end{proof}

\section{Masas of $\QQ$ in the presence of the Continuum Hypothesis.}

This section is devoted to the proof of the main ingredient of item (B) of Theorem \ref{main} which is
Proposition \ref{ch-proposition}. After a  simple Lemma \ref{commutator}
we prove the main technical lemma (\ref{substitution}) concerning $\bb$ which after passing
to $\QQ$ gives us the main lemma (\ref{main-lemma}). Given a countable Boolean algebra $\A$
of projections in $\QQ$ and a projection $R\in \QQ$ commuting with $\A$ and a self-adjoint $S\in\QQ$ commuting with $\A$
such that $R$ and $S$ split the same ultrafilter of $\A$ (see Definitions \ref{def-ba-split} and \ref{def-split}),
Lemma \ref{main-lemma} allows us to replace $R$ with $R'$ which has the same position
with respect to $\A$ as $R$ but $R'$ does not commute with $S$. 
This is fundamental for the transfinite induction in the proof of Proposition \ref{ch-proposition},
where we adopt a general scenario from \cite{pathology} while replacing the use of the Voiculescu
Bicommutant Theorem by Lemma \ref{main-lemma}.

\begin{lemma}\label{commutator} Let $T\in \bb$  be 
self-adjoint and $c_1<c_2$ satisfy $c_1, c_2\in \sigma(T)$.
Let $\E=(e_n)_{n\in \N}$ be an orthonormal basis of $\ell_2$.  
Then for every $\varepsilon>0$ there is a finite $F\subseteq \N$,
and a projection  $P\in \bb$ 
 such that
$$\|[P_FTP_F, P]\|>{{c_2-c_1\over2}-\varepsilon}$$
and $\langle P(e_n), e_m\rangle\not=0$ implies $n, m\in F$ for every $n,m\in \N$,
where $P_F$ is the projection in $\bb$ onto the space spanned by $\{e_n: n\in F\}$.
\end{lemma}
\begin{proof} 
First note that for $i=1,2$ there is no $\varepsilon>0$ such that the operator $T-c_iI$ is
 bounded below by $\varepsilon$, i.e., $\|(T-c_iI)(x)\|\geq \varepsilon\|x\|$ for all $x\in \ell_2$. This is because otherwise
 $T-c_iI$ is injective and has closed range, and 
$0\not=\langle (T-c_iI)y, (T-c_iI)y \rangle=\langle (T-c_iI)^2y, y \rangle$ for every $y\not=0$, so
$T-c_iI$ is also surjective and so invertible, contradicting the hypothesis that $c_i\in \sigma(T)$.

Fix $\varepsilon>0$. By passing to a smaller $\varepsilon$ we may assume that 
$\varepsilon\leq(c_2-c_1)/2$. By the above there are unit vectors  $x_1', x_2'\in \ell_2$ such that 
$\|T(x_i')-c_ix_i'\|<\varepsilon$ for $i=1,2$.
By considering finitely supported (with respect to $\E$) unit approximations $x_1, x_2$ to $x_1', x_2'$
so that  $T(x_1), T(x_2)$ approximate $T(x_1'), T(x_2')$ 
we can find  a finite $F\subseteq\N$ such that $x_1, x_2\in span(\{e_n: n\in F\})$
and  $\|P_FTP_F(x_i)-c_ix_i\|<\varepsilon$ for $i=1,2$.

Now we will work in  $span(\{e_n: n\in F\})$.
Find an orthonormal basis $f_1, \dots f_n$ for $span(\{e_n: n\in F\})$ with respect to which $P_FTP_F$ is diagonal with (real)
diagonal entries $a_1, \dots, a_n$.
Note that for any $\lambda\in \R$ and $i=1,2$ we have 
$$\alpha(\lambda)=\max_{1\leq j\leq n}|a_j-\lambda|=\|P_FTP_F-\lambda P_F\|\geq\|P_FTP_F(x_i)-\lambda x_i\|=$$
$$=\|c_ix_i+P_FTP_F(x_i)-c_ix_i-\lambda x_i\|>
 |c_i -\lambda|-\varepsilon.\leqno (*)$$
Let $a_j$ and $a_k$ be the smallest and the largest values of $a_1, \dots, a_n$ respectively
and $\lambda=(a_j+a_k)/2$. Then 
$$\alpha(\lambda)=|a_j-\lambda|=|a_k-\lambda|=(a_k-a_j)/2\leqno (**)$$
and so the triangle inequality,  (*) and (**) yield
$$c_2-c_1-2\varepsilon\leq|c_2-\lambda|-\varepsilon+|c_1-\lambda|-\varepsilon<
\alpha(\lambda)+\alpha(\lambda)= a_k-a_j.\leqno(***)$$
Let   $P$ be the orthogonal projection onto the subspace of $\ell_2$
spanned by  $w={1\over\sqrt{2}}f_k+{1\over\sqrt{2}}f_j$.
We have 
$P(f_j)={1\over\sqrt{2}}w={1\over{2}}f_k+{1\over{2}}f_j$, so
$P_FTP_FP(f_j)={a_k\over{2}}f_k+{a_j\over{2}}f_j$, while
$PP_FTP_F(f_j)=P(a_jf_j)=a_jP(f_j)={a_j\over{2}}f_k+{a_j\over{2}}f_j$, and so putting $v=f_j$ by (***)
$$\|(P_FTP_FP-PP_FTP_F)(v)\|=\|{a_k-a_j\over{2}}f_k\|={a_k-a_j\over 2}> {c_2-c_1\over 2}-\varepsilon.$$

So  $P$ is as required.

\end{proof}

\begin{lemma}\label{substitution} Let $\E$ be an orthonormal basis of $\ell_2$.
Suppose that 
\begin{enumerate}
\item $\A\subseteq\B\subseteq \D(\E)$ are countable Boolean subalgebras of projections,
\item  $\B$ is generated by $\A\cup\{P\}$, where $P\in \D(\E)$ is a projection,
\item $\pi|\B$ is a Boolean isomorphism,
\item $\U$ is an ultrafilter in $\A$ which is split by $P$,
 \item $T$ is a self-adjoint element of $\bb$ such that for some reals $c_1<c_2$
 $$c_1, c_2\in \sigma_{S\bb S}(STS)$$
  for every  $R\in \U$ and  every projection $S\in\bb$ such that $R-S\in\K$,
where the spectrum is calculated in the corner $S\bb S$ whose unit is
$S$.
\end{enumerate}
Then there is  a projection $Q$ in $\bb$  and there is a Boolean
monomorphism $h: \pi[\B]\rightarrow \QQ$  such that
 for every $R\in \U$ we have
\begin{enumerate}
\item[(6)] $[T, Q]\not\in \K.$
\item[(7)] $RQ, QR, R(I-Q), (I-Q)R \not\in\K,$
\item[(8)] $(P-Q)(I-R), (I-R)(P-Q)\in \K,$
\item[(9)] $[Q, A]\in \K$ for every $A\in \A$, 
\item[(10)] $h$ is the identity on $\pi[\A]$
and $h(\pi(P))=\pi(Q)$.
\end{enumerate}
\end{lemma}
\begin{proof} Let $\E=(e_k)_{k\in \N}$. For $A\subseteq \N$ by $P_A$ we will mean the projection in $\D(\E)$ such that
$P_A(e_k)=e_k$ if $k\in A$ and $P_A(e_k)=0$ if $k\not\in A$. All projections in $\D(\E)$ are of this form for
some $A\subseteq \N$, in particular $P=P_X$ for some $X\subseteq \N$. 
Let $(A_n)_{n\in \N}$ be such a sequence of subsets of $\N$ that
$A_0=\N$ and
$\U=\{P_{A_n}: n\in \N\}$. 

By induction on $n\in \N$ we will construct $(F_n)_{n\in \N}$, $(Q_n)_{n\in \N}$, 
such that for every distinct $n, n'\in \N$ the following hold:
\begin{enumerate}
\item[(a)] $F_n$ is a finite subset of $\N$, 
\item[(b)] $F_n\cap F_{n'}=\emptyset$,
\item[(c)] $F_n \subseteq \bigcap_{l\leq n}A_l$,
\item[(d)] $Q_n\in \bb$ is  a projection  such that $\langle Q_n(e_k), e_{k'}\rangle\not =0$ implies $k, k'\in F_n$
for all $k, k'\in \N$,
\item[(e)] $\|[P_{F_n}TP_{F_n}, Q_n]\|> (c_2-c_1)/3$.
\end{enumerate}
Once we have such objects we put 
$$B=\N\setminus\bigcup_{n\in \N} F_n,$$ and we define 
$$ Q(e_k) =
  \begin{cases}
    P(e_k) & \text{if $k\in  B$}, \\
    Q_n(e_k) & \text{if $k\in F_n,\ n\in \N$}.
  \end{cases}
  $$
First let us present an inductive construction satisfying (a) - (e). Suppose that we are done up to $n\in\N$.
  Let $A'=\bigcap_{l\leq n}A_l$ and $A=A'\setminus \bigcup_{k<n}F_k\subseteq \N$.  Note that $A$
  is infinite because $A'$ is infinite because it has to be nonempty since it is in an
  ultrafilter of $\A$, but the nonemptiness implies being infinite since $\pi(P_{A'})\not=0$ by (3).
  $P_{A'}\in \U$  and  $P_{A}$ is its compact perturbation by (a)
  so $c_1, c_2\in \sigma_{P_A\bb P_A}(P_ATP_A)$ by hypothesis (5). Apply Lemma \ref{commutator} for the operator
   $P_ATP_A$ and the Hilbert space generated by its orthonormal basis $\{e_k: k\in A\}$ finding
   $F_n, Q_n$ as in (a) - (e). 
  
  Before going to the proof of (6) - (10) let us note that for
  every $R\in \U$ we have 
  \begin{enumerate}
  \item[(f)] $QP_{F_n}=P_{F_n}QP_{F_n}=Q_n$ for all $n\in \N$,
  \item[(g)] $Q_n=P_{F_n}Q$ for all $n\in \N$,
  \item[(h)] $RP_{F_n}=P_{F_n}R=P_{F_n}$ for all but finitely many $n\in \N$, 
  \item[(i)] $RQP_{F_n}=RP_{F_n}Q=P_{F_n}RQ=Q_n=RQ_n=Q_nR$ for all but finitely many $n\in \N$.
  \end{enumerate}
  Item (f) follows from (b), (d) and from the definition of $Q$.
 
 For (g)  note that 
 $P_{F_n}QP_{\N\setminus F_n}=0$ by the definition of $Q$ by (b), (d) 
 and by the fact that $P\in \D(\E)$. So 
 by (f)  we have  $P_{F_n}Q=P_{F_n}QP_{F_n}=Q_n$.
 
For (h) $R=P_{A_{n'}}$ for some $n'\in \N$ and so by (c) for $n\geq n'$ we have 
  (h).  Items (f) and  (g)  and (h) imply (i).
  
  Now let us prove that $Q$ satisfies (6) - (10). We will often be using an elementary
  fact that $UV\in \K$ if and only if $VU\in \K$ for self-adjoint $U, V\in \bb$ which follows
  from the Schauder theorem and the fact that $VU=V^*U^*=(UV)^*$. So, for example, only half of (7) and (8)
  needs to be proved.

  First let us prove (6). For $n\in\N$ let $v_n\in span(\{e_k: k\in F_n\})$ be a unit vector satisfying
  $$\|[P_{F_n}TP_{F_n}, Q_n](v_n)\|> (c_2-c_1)/3.$$
   It exists by (d) and (e).
   By (f), (g) we have
   $$P_{F_n}[T, Q]P_{F_n}=P_{F_n}TQP_{F_n}-P_{F_n}QTP_{F_n}=
   P_{F_n}TP_{F_n}Q_n-Q_nP_{F_n}TP_{F_n}=$$
   $$= [P_{F_n}TP_{F_n}, Q_n].$$
  We conclude that 
  $$\|[T, Q](v_n)\|\geq \|P_{F_n}[T, Q]P_{F_n}(v_n)\|= \|[P_{F_n}TP_{F_n}, Q_n](v_n)\|>(c_2-c_1)/3,$$ and so
  $[T, Q]\not\in \K$  as needed for (6), since the $v_n$s are pairwise orthogonal as they are disjointly supported.
  
  For (7) since $Q_n$s are projections which do not commute with $P_{F_n}TP_{F_n}$ by (d), (e) we conclude
  that they are not $0$ nor $P_{F_n}$ and so $\|Q_n\|=1$ and $\|P_{F_n}-Q_n\|=1$
   So for $n\in \N$ there are $F_n$ supported unit vectors $w_n$ and $u_n$ witnessing these facts respectively.
   And so for all but finitely many $n\in\N$ we have
   $$\|RQ(w_n)\|=\|RQP_{F_n}(w_n)\|=\|Q_n(w_n)\|=1$$ by (i) and 
   $$\|R(I-Q)(u_n)\|=\|R(I-Q)P_{F_n}(u_n)\|=\|R(P_{F_n}-Q_n)(u_n)\|=\|(P_{F_n}-Q_n)(u_n)\|=1$$
   by (h), (f)  and (i) as needed for (7), since the $w_n$s  and $u_n$s
   are pairwise orthogonal respectively as they are disjointly supported.

 To prove (8) note that if $R\in \U$ there is $k\in \N$ such that $I-R=P_{\N-A_k}$ and 
 only finitely many $F_n$s intersect $\N\setminus A_k$ by (c). It follows that
 $Q(I-R)$ is a finite dimensional perturbation of $P(I-R)$ and so
 $(P-Q)(I-R)\in\K$
  
For (9) let $A\in \A$. Since $\U$ is an ultrafilter of $\A$, $A=I-R$ for $R\in\U$
or $A=R\in \U$.
In the first case 
$$[Q, A]=[Q-P+P, I-R]=[Q-P, I-R]+[P, I-R]\in \K,$$
since $[Q-P, I-R]\in \K$ by (8) and $[P, I-R]=0$ as $P\in \B\subseteq \D(\E)$.
In the second case 
$$[Q, A]=[Q, I-(I-R)]=[Q, I]-[Q, I-R]\in \K$$
since $[Q, I]=0$ and $[Q, I-R]\in \K$ by the first case.
 This completes the proof of (9).

Note that by (9) the set $\pi[\A]\cup\{\pi(Q)\}$ generates a Boolean algebra $\CC$ of projections of $\QQ$.
To construct $h: \pi[\B]\rightarrow \CC\subseteq\QQ$ as in (10)
we will use Sikorski's extension criterion (Theorem 5.5 of \cite{koppelberg})
which implies that the identity on $\pi[\A]$ can be extended to  an isomorphism $h$ defined on one element extension
$\pi[\B]$ and satisfy $h(\pi(P))=\pi(Q)$ if and only if the following two conditions hold for
all $A\in \A$:
\begin{enumerate}
\item[(i)] $\pi(A)\pi(P)=0 \Leftrightarrow \pi(A)\pi(Q)=0$.
\item[(ii)] $\pi(A)(1-\pi(P))=0 \Leftrightarrow \pi(A)(1-\pi(Q))=0$.
\end{enumerate}
First note that $\pi(A)\pi(P)=0$ or $\pi(A)(1-\pi(P))=0$ for $A\in \A$ only if
$A=I-R$ for $R\in \U$, because otherwise $A\in\U$ as $\U$ is an ultrafilter of $\A$
but this would contradict the hypothesis (4) that $P$ splits the ultrafilter $\U$ as $\pi$ is an isomorphism
on $\B$ by (3), so $\pi(P)$ splits $\pi[\U]$.
Also $\pi(A)\pi(Q)=0$ or $\pi(A)(1-\pi(Q))=0$ for $A\in \A$ only if
$A=I-R$ for $R\in \U$ by (7) as $\U$ is an ultrafilter of $\A$. So in the proof of (i) and (ii) we may assume that 
$A=I-R$ for $R\in \U$. But in this case by (8) we have
$\pi(A)\pi(P)=\pi(A)\pi(Q)$ and so $\pi(A)-\pi(A)\pi(P)=\pi(A)-\pi(A)\pi(Q)$
and so (i) and (ii) are proved, as needed for (10).
\end{proof}

\begin{lemma}\label{main-lemma} Suppose that 
\begin{enumerate}
\item $\A\subseteq \B\subseteq \QQ$ are countable Boolean  algebras of projections,
\item $\B$ is generated over $\A$ by  a projection $R\in \B\setminus\A$, 
\item $T\in \bb$ is  self-adjoint and $\pi(T)$ commutes with all elements of $\A$ but does not
belong to  $C^*(\A)$,
\item $R$  splits ultrafilter $\U$ of $\A$ and $T$ C*-splits $\U$ as well.
\end{enumerate}
Then there is a projection $R'\in \QQ$ and a Boolean monomorphism $g: \B\rightarrow \QQ$ such that
\begin{enumerate}
\item[(5)] $R'$ commutes with $\A$
\item[(6)] $g$ restricted to $\A$ is the identity, $g(R)=R'$
\item[(7)] $R'$ does not commute with $\pi(T)$.
\end{enumerate}
\end{lemma}
\begin{proof}
First use Lemma \ref{calkin-lift} to find an orthonormal basis $\E$ and
Boolean algebras $\B', \A'$ of projections included in $\D(\E)$ such that
$\pi|\B'$ is a Boolean isomorphism onto $\B$, with $\pi[\A']=\A$. Let $\U'$ be
an ultrafilter of $\A'$ such that   $\pi[\U']=\U$ and $P$ be a projection in $\B'$ such that
$\pi(P)=R$.  So items (1) - (4) of lemma \ref{substitution}
are satisfied  for $\A', \B', P, \U'$.  
Now we show that the hypothesis (5) is satisfied as well. Let $A'\in \U'$ and $S\in \bb$
be a projection such that $A'-S\in \K$.
Since $T$ C*-splits $\U$ by Lemma \ref{c-split} there are reals
$c_1<c_2$ such that 
$$c_1, c_2\in \sigma_{\pi(A')\QQ \pi(A')}(\pi(A')\pi(T)\pi(A'))$$ 
as $\pi(A')\in \U$.
Since 
$$\pi(S T S)=\pi(A')\pi(T)\pi(A')$$ and
$\pi[S\bb S]=\pi[A'\bb A']=\pi(A')\QQ\pi(A')$ and since 
the spectrum cannot grow in quotients (1.4.2, 2.5.7 of \cite{ilijas-book})
we have $c_1, c_2\in \sigma_{S\bb S}(STS)$. So item (5) of the hypothesis of Lemma \ref{substitution}
is satisfied as well. Hence we can apply Lemma \ref{substitution} obtaining $h:\B\rightarrow \QQ$ and $Q\in \bb$.
Now we note that $g=h$ and $R'=\pi(Q)$ are
the objects required in (5) - (7) of Lemma \ref{main-lemma}.
\end{proof}

\begin{proposition}[{\sf CH}]\label{ch-proposition} Suppose that $K$ is totally disconnected, compact, 
Hausdorff,  does not admit a $G_\delta$-point
and is of weight $\omega_1=\mathfrak c$.  
Then there is a collection of more than $\cc$ masas in $\QQ$ which are
$*$-isomorphic to $C(K)$, do  not have  commutative lifts and are pairwise unitarily non-equivalent.
\end{proposition}
\begin{proof}
In fact the nonexistence of a commutative lift follows for all but three $K$s (listed in items (1) - (3)
in the Introduction) from the homeomorphic type of  $K$ by Lemma \ref{masa-masaQQ}.
Assume {\sf CH}.  
Let
\begin{itemize}

\item $X, Y, Z$ be a partition of $\omega_1$ into three uncountable sets,
\item $\A$ be the Boolean algebra of all clopen subsets of $K$ and
\item  
$\{A_\alpha: \alpha\in X\}$ be an enumeration of all elements of $\A$,
\item $\{T_\alpha:\alpha\in Y\}$ be an enumeration of all self-adjoint elements of
$\bb$. 
\item $\{\CC_\alpha: \alpha\in Z\}$ be an enumeration of  some fixed family of
masas of $\QQ$,
\end{itemize}
The possibility of the enumerations as in  the last two items follows from the hypothesis  
that $K$ is totally disconnected and of weight $\omega_1$ and    from {\sf CH}.
First we will show how to construct a masa $\CC$ of $\QQ$ which is $*$-isomorphic to
$C(K)$ and $\CC\not=\CC_\alpha$ for all $\alpha\in Z$.

By induction on $\alpha<\omega_1$ we construct projections $B_\alpha\in \QQ$, elements
$A_\alpha'\in \A$, Boolean isomorphisms $h_\alpha$  such that 
\begin{enumerate}
\item The projections $\{B_\beta: \beta<\alpha\}$ commute and so they generate a Boolean algebra 
denoted $\B_\alpha\subseteq \QQ$.
\item $h_\alpha: \A_\alpha'\rightarrow \B_\alpha$ satisfies $h_\alpha(A_\beta')=B_\beta$ for all $\beta<\alpha$, where
$\A_\alpha'$ is the Boolean subalgebra of $\A$ generated by $\{A_\beta': \beta<\alpha\}$.
\item If $\alpha \in X$, then 
$A'_\alpha=A_\beta$,  where $\beta=\min\{\gamma\in \alpha\cap X: A_\gamma\not\in \A_\alpha'\}$ if this set is nonempty.
\item If $\alpha\in Y$ and  $\pi(T_\alpha)\not\in C^*(\B_\alpha)$ but $\pi(T_\alpha)$  commutes with all elements of
$\B_\alpha$,
then   $B_\alpha$ does not commute with $\pi(T_\alpha)$, 
\item If $\alpha\in Z$ and  $h_\alpha[\A_\alpha']\subseteq \CC_\alpha$, then there is $R\in \CC_\alpha$
 such that   $B_\alpha$ does not commute with $R$, 
\end{enumerate}
First note that  (1) and (2) imply that  $\A_\beta'\subseteq \A_\alpha'$,  
$\B_\beta\subseteq \B_\alpha$, $h_\beta\subseteq h_\alpha$
for all $\beta<\alpha<\omega_1$. Put $\A'=\bigcup_{\alpha<\omega_1}\A_\alpha'$,
$\B=\bigcup_{\alpha<\omega_1}\B_\alpha$, $h=\bigcup_{\alpha<\omega_1}h_\alpha$.
Condition (3) implies that $\A=\A'$ and $h:\A\rightarrow\B$ is a Boolean isomorphism
and so by the Stone duality the Stone space $K_\A$ of $\A$ is homeomorphic to
the Stone space $K_\B$ of $\B$ and so by Lemma \ref{cstar-ba-iso} 
the algebra $\CC=C^*(\B)$ is $*$-isomorphic to $C(K_\A)$.
Condition (4) and Lemma \ref{sa-masa} imply that 
 $\CC=C^*(\B)$ of $\QQ$ generated by $\B$ is a masa of $\QQ$. Condition (5) 
 implies that $\CC\not=\CC_\alpha$ for any $\alpha\in Z$.
 So we are left with showing how to do the inductive step $\alpha$.

Assume that we have (1) - (5) up to $\alpha<\omega_1$. 
If $\alpha=0$, then $\A_0'$ and $\B_0$ are isomorphic to
the trivial Boolean algebra $\{0,1\}$ and so  let $h_0:\A_0'\rightarrow\B_0$ be the unique Boolean isomorphism.
If $\alpha$ is a limit ordinal put $h_\alpha=\bigcup_{\beta<\alpha}h_\beta$.
If $\alpha$ is a successor ordinal, $h_\alpha$ has already been defined at the preceding step.

First consider the case of $\alpha\in X$. Let $\beta$ be as in (3) and assume the set 
in (3) is nonempty. Then we put $A_\alpha'=A_\beta$ as in (3).
We need to find a Boolean monomorphism $h_{\alpha+1}: \A_{\alpha+1}'\rightarrow \QQ$
which extends $h_\alpha$ and satisfies (2). This can be done by Lemma \ref{injective-Q}.
Finally we put $B_\alpha=h_{\alpha+1}(A_\alpha')$. So we obtain (1) and (2). If $\alpha\in X$ and the set
from (3) is empty we choose $A_\alpha'$ to be any element of 
$\A\setminus \A_\alpha'$ and proceed as in the nonempty subcase.
Conditions (4)  and (5) are void in this case.

If $\alpha\in Y$ and $\pi(T_\alpha)$ does not commute with the entire $C^*(\B_\alpha)$ or
$\pi(T_\alpha)\in C^*(\B_\alpha)$, then we choose $A_\alpha'\in\A\setminus\A_\alpha'$ and proceed as in the case of $\alpha\in X$.
Otherwise we have $\pi(T_\alpha)$ in $\QQ\setminus C^*(\B_\alpha)$
 commuting with the entire $C^*(\B_\alpha)$. This means that  $C^*(\B_\alpha\cup\{\pi(T_\alpha)\})$ is commutative
 as $T_\alpha$ is self-adjoint
 and $C^*(\B_\alpha)$ is a proper C*-subalgebra of $C^*(\B_\alpha\cup\{\pi(T_\alpha)\})$. So by Lemma \ref{split-point} there is
 an ultrafilter $\U$ of $\B_\alpha$ which is C*-split by $T_\alpha$. 
 As $K$ does not admit any $G_\delta$-point,  by Lemma \ref{gdelta-split} there is $\gamma\in X$ such that
 $A_\gamma\in \A\setminus\A_\alpha'$ and such that $A_\gamma$
 splits $h_{\alpha}^{-1}[\U]$. Define $A_\alpha'=A_\gamma$.
 Use Lemma \ref{injective-Q} to find a Boolean monomorphism $h_{\alpha+1}': \A_{\alpha+1}'\rightarrow \QQ$
 which extends $h_\alpha$. So $h_{\alpha+1}'(A_\alpha')$ splits $\U$. 
 Now 
  we are in the position to use
 Lemma \ref{main-lemma} for  the algebras $\B_{\alpha}\subseteq \B_{\alpha}'$
 where $\B_\alpha'$ is generated by $\B_\alpha$ and $h_{\alpha+1}'(A_\alpha')$
 for $R=h_{\alpha+1}'(A_\alpha')$ and $T=T_\alpha$.
 
So we obtain $R'\in \QQ$ which commutes with $\B_\alpha$ and a Boolean monomorphism
$g: \B_{\alpha}'\rightarrow\QQ$ which extends the identity on $\B_\alpha$
and satisfies $g(h_{\alpha+1}'(A_\alpha'))=R'$ while 
$R'$ does not commute with $\pi(T_\alpha)$.

So put $B_{\alpha}=R'$  and $h_{\alpha+1}=g\circ h_{\alpha+1}'$.
The properties of $g$ imply that $h_{\alpha+1}\supseteq h_\alpha$ and 
$h_{\alpha+1}(A_{\alpha}')=R'=B_{\alpha}$ as required in (2).
Condition (1) is satisfied by Lemma \ref{main-lemma} (5).
Conditions (3) and (5)  are void in this case.
Condition (4) follows from Lemma \ref{main-lemma} (7).

Finally consider the case of $\alpha\in Z$. If
$h_\alpha[\A_\alpha']=\B_\alpha\not\subseteq \CC_\alpha$, then we choose
$A_\alpha'\in\A\setminus\A_\alpha'$ and proceed as in the  case of $\alpha \in X$.
So assume that $h_\alpha[\A_\alpha']=\B_\alpha\subseteq \CC_\alpha$.
Since each masa of $\QQ$ is nonseparable
(e.g. by Proposition \ref{gdelta}) and $\A_\alpha'$ is countable, we may find a 
self-adjoint $R_\alpha\in \CC_\alpha\setminus C^*(\B_\alpha)$.
As $\B_\alpha\cup\{R_\alpha\}\subseteq \CC_\alpha$, the element $R_\alpha$
commutes with the entire $C^*(\B_\alpha)$. So we may proceed as in the case of $\alpha\in Y$ with 
a self-adjoint $T'$ satisfying $\pi(T')=R_\alpha$ in
the place of $T_\alpha$. 

This completes the inductive step of (1) - (5)  and by the argument from the paragraph after the
conditions (1) - (5) this proves the existence
of  a masa $\CC$ of $\QQ$ which is $*$-isomorphic to
$C(K)$ and $\CC\not=\CC_\alpha$ for all $\alpha\in Z$.
To make sure that there is such a masa which does not have a lift to $\bb$
while carrying out the construction as in (1) - (5) consider 
$$\{\CC_\alpha: \alpha\in Z\}=\{\pi[\mathcal D]: \mathcal D \ \hbox{is a masa of}\  \bb\}.$$
This is possible since the right hand side above  has cardinality $\cc$ 
as there are exactly $\cc$  masas of $\bb$ (Proposition 12.3.1 of \cite{ilijas-book}).
Now Lemma \ref{masa-masaQQ} implies that $\CC$ constructed as above
for this $\{\CC_\alpha: \alpha\in Z\}$ does not have a commutative lift.
To obtain more than $\cc=\omega_1$ such pairwise unitarily 
non-equivalent masas $\mathcal D_\xi$ of $\QQ$ for $\xi<\omega_2$
construct them by transfinite induction on $\xi<\omega_2$ performing the above construction
satisfying (1) - (5) for 
$$\{\CC_\alpha: \alpha\in Z\}=\{\pi[\mathcal D]: \mathcal D \ \hbox{is a masa of}\  \bb\}
\cup\{U\mathcal D_\eta U^*: U \in \mathbb U(\QQ), \ \eta<\xi\},$$
where $\mathbb U(\QQ)$ stands for the set of all unitary elements of $\QQ$ which has
cardinality $\cc$ as $\bb$ has cardinality $\cc$.
\end{proof}

\section{Masas of $\QQ$  with no additional set-theoretic hypothesis}

In this section we do not use any additional set-theoretic hypothesis. First, in Proposition \ref{gdelta}
 we prove that the Gelfand space of a masa of $\QQ$ cannot admit a $G_\delta$ point. This is related to
 item (B) of Theorem \ref{main}. The remainder of this section is devoted to proving item (A) of Theorem \ref{main}.
 
To do so we start preparatory work related to viewing $\bb$ as $\bb\otimes M_2$ 
(Definition \ref{def-Azero}, Lemma \ref{limit-compact})
as our masas will be included in $\pi[\D(\E)\otimes M_2]$, where $\D(\E)$ is the algebra of all diagonal
operators with respect to a fixed orthonormal basis $\E$. They will also include $\pi[\D(\E)\otimes \C I_2]$.
In fact, it will be more convenient to work with preimages of masas under $\pi$, than with
masas of $\QQ$, we call such preimages almost masas in Definition \ref{def-almost}. Lemma \ref{crucial} is crucial for our
almost masas and masas and its proof is based on the Johnson and Parrott theorem. 

In Definition \ref{def-Ds} and Lemmas \ref{71/72} and \ref{dx} we deal with choosing
orthonormal bases in the algebra $M_2$  of two by two matrices. Finally in Definition 
\ref{def-M} we define the C*-algebras some of which will witness item (A) of Theorem \ref{main}. 
They depend on  two parameters: the family $\X$ of subsets of $\N$ and the family of choices of reals 
$\mathfrak A=(\alpha_X: X\in \X)$, where $\alpha_X\in (71/72, 1]$. The reals $\alpha_X$
represent orthonormal  bases in $M_2$.  We define {\sl wide} families $\X$ in Definition \ref{def-wide}
and {\sl coherent} choices $\mathfrak A$ in Definition \ref{def-coherent}.
The algebras $\M_{\X, \mathfrak A}$ are almost masas (in particular $\pi[\M_{\X, \mathfrak A}]$ is a masa in $\QQ$)
for $\X$ wide and $\mathfrak A$ coherent as proved in Lemma \ref{almost-masa}.
Two following lemmas (\ref{splitting}, \ref{gen-projections}) prove that our almost masas
$\M_{\X, \mathfrak A}$ are generated by projections, in particular
the masas $\pi[\M_{\X, \mathfrak A}]$ are generated by projections as well.

Definition \ref{def-BX} and Proposition \ref{few-projections}  concern
particular type of the masa $\pi[\M_{\X, \mathfrak A}]$, where
the wide family $\X$ is a maximal almost disjoint family. We are able to 
explicitly describe the Gelfand space of the masa in Proposition \ref{few-projections} (3). 
The rest of the section is devoted to showing that among our masas $\pi[\M_{\X, \mathfrak A}]$ there are
$2^\cc$ pairwise non-$*$-isomorphic ones. 

\begin{proposition}\label{gdelta} Suppose that $\A$ is a masa in the Calkin algebra and $\A$ is isomorphic to $C(K)$, where $K$
is compact and Hausdorff. Then $K$ does not admit a $G_\delta$ point.
\end{proposition}
\begin{proof} 
Let $\iota : C(K)\rightarrow \QQ$ be a $*$-embedding onto its range $\A$.
Suppose that there are open $U_n\subseteq K$ such that $U_0=K$,  $\overline{U_{n+1}}\subseteq U_n$ for each $n\in \N$ and
$\bigcap U_n=\{x\}$ for some $x\in K$.  The compactness of $K$ implies that 
$\{U_n\}_{n\in \N}$ is a neighborhood basis at $x$. 

We will show that $\A=\iota[C(K)]$ is not a masa.
Consider $g=\sum_{n\in \N}f_n/2^{n+1}\in C(K)$, where 
$f_n\in C(K)$ is real-valued, $0\leq f_n\leq 1$  and 
satisfies $f_n|(K\setminus U_n)=1$ and $f_n|\overline{U_{n+1}}=0$ for each $n\in \N$. 
Note that $\{x\}=g^{-1}[\{0\}]$.  As $\{U_n\}_{n\in \N}$ is a basis at $x$,
 whenever $h\in C(K)$ is zero on a neighborhood  of $x$, 
there is $n\in \N$ such that
$h|U_n=0$ and $(1/g)|(K\setminus U_n)$ extends to a function $g_n\in C(K)$. So 
$$h=gg_nh=hg_ng.\leqno (*)$$

Let $T\in \bb$ be self-adjoint such that $\pi(T)=\iota(g)$. By the Weyl-von Neumann theorem 
there is an orthonormal basis $\E$ of $\ell_2$ and $T', S\in \bb$ such that  $T=T'+S$ and
$T'\in \D(\E)$ and $S\in \K$. The image $\pi[\D(\E)]$ is $*$-isomorphic to
$C(\N^*)$ by Lemma \ref{pi} and contains $\iota(g)$. Let $\phi: C(\N^*)\rightarrow \pi[\D(\E)]$ be the $*$-isomorphism.

We will find a self-adjoint $R\in \QQ$ such that $R\not \in \A$ but $R$ commutes with the entire $\A$.
This will  show that $\A$ is not a masa of $\QQ$ by Lemma  \ref{sa-masa}.

 The spectrum
of $g$ contains $0$ so the same holds for $G=\phi^{-1}(\iota (g))\in C(\N^*)$ as the spectrum does not change 
when we pass to a bigger C*-algebra (1.11.6 of \cite{ilijas-book}).  So $G^{-1}[\{0\}]$ is a nonempty
$G_\delta$-set ($G$ as a continuous real-valued function on $\N^*$) 
 in $\N^*$ and such sets have nonempty interior (Lemma 1.2.3 of \cite{handbook-van-mill}).
So there is a family $\F$ of cardinality continuum  (corresponding to an almost disjoint family of subsets of $\N$)
of pairwise disjoint nonempty clopen sets 
$U\subseteq G^{-1}[\{0\}]$.
For all $U\in \F$ we have  
 $$\iota(g)\phi(1_U)=\phi(G\cdot 1_U)=\phi(0)=0.\leqno (**)$$
 If $g'\in C(K)$ is a nonzero projection  and $gg'=0$, then $g'=1_{\{x\}}$. 
So at most one of the functions $\phi(1_U)$ for $U\in \F$ may be in $\A$, namely the one 
satisfying $\phi(1_U)=\iota(1_{\{x\}})$. Take some other $U\in \F$
so that $\phi(1_U)\not\in \A$ (so we only use the fact that $\F$ has at least two elements).

 By (*) and (**), when $h\in C(K)$ is zero on a  neighborhood of $x$, then 
 $\iota(h)\phi(1_U)=\iota(hg_n)\iota(g)\phi(1_U)=0$. Of course
 $\phi(1_U)$ commutes with all multiples of the unit in $\QQ$ which is the unit in $\A$
 and is equal to $\iota(1_K)$. However,
  the collection of all elements of $C(K)$ which are either constant or are zero on a neighborhood
of $x$ is self-adjoint and generates a norm dense algebra in $C(K)$ by the Stone-Weierstrass theorem. It follows that 
the self-adjoint $R=\phi(1_U)$
commutes with all elements of $\A$ and  $R\not\in \A$ so $\A$ is not a masa of $\QQ$.

\end{proof}

The remainder of this section is devoted to a construction of many $*$-isomorphically distinct 
masas in $\QQ$ without making an additional set-theoretic assumption. For the remainder of this section
fix an orthonormal basis $\E=(e_k)_{k\in \N}$ of $\ell_2$. 

\begin{definition}\label{def-Azero} $ $
\begin{enumerate}
\item $P_A\in \bb$ is the orthogonal projection onto the closed subspace of $\ell_2$
spanned by $\{e_{2k}, e_{2k+1}: k\in A\}$ for $A\subseteq \N$.
\item $\BB_n=\B(\ell_2(\{2n, 2n+1\}))$.
\item $I_n\in \BB_n$ is the unit of $\BB_n$
\item  $\A_0$ is the set of all operators $T\in \bb$ satisfying  for all $k, m\in\N$
$$\langle T(e_k), e_m\rangle \not=0 \ \Rightarrow\  \exists n\in \N \ \{k, m\}
\subseteq \{2n, 2n+1\}.$$
\item If $T\in \A_0$ and $n\in \N$, then $T_n\in \BB_n$
is given by $T_n(e_{2n+i})=T(e_{2n+i})$ for $i=0,1$. 
\item In the circumstances as in (5) we will write $T=(T_n)_{n\in \N}$.
\item $\D_0$ is the set of all operators $T=(T_n)_{n\in \N}\in \A_0$ for which there is a sequence
$(a_n)_{n\in \N}\in \ell_\infty$ such that $T_n=a_nI_n$
\item If $X\subseteq \N$, then $\A_0(X)$ is the set of all elements $T=(T_n)_{n\in \N}$ of $\A_0$
which satisfy $T_n=0$ for all $n\in \N\setminus X$. 
\item $\D_0(X)=\A_0(X)\cap \D_0$, where $X\subseteq \N$.

\end{enumerate}
\end{definition}

\begin{lemma}\label{limit-compact} Suppose that $T=(T_n)_{n\in \N}\in \A_0$. Then $T$ is compact if and only
if $$\lim_{n\rightarrow \infty}\|T_n\|=0.$$
\end{lemma}
\begin{proof}
If the limit is not zero, then there are infinitely many $n\in \N$ and $\varepsilon>0$ such that
$\|T_n(v_n)\|>\varepsilon$ for some unit vectors $v_n\in \ell_2(\{2n, 2n+1\})$. As such 
$v_n$s are pairwise orthogonal we conclude that $T$ is not compact.

Now suppose that the limit is zero. Let $k\in \N$ be such that $\|T_n\|\leq\varepsilon$
for all $n>k$. It is enough to show that if 
 $v\in \ell_2(F)$ is of norm one for finite $F$ satisfying $\min(F)>2k+1$, then $\|T(v)\|\leq \varepsilon$.
 Such $v$ is of the form $v=\sum_{n\in G}v_n$, where $v_n\in \ell_2(F\cap \{2n, 2n+1\})$ and
 $G=\{n\in \N: F\cap \{2n, 2n+1\}\not=\emptyset\}$ with 
 $\sum_{n\in G}\|v_n\|^2=1$.  So as $T_n$s have orthogonal
 ranges for distinct $n\in \N$ we have 
 $$\|T(v)\|^2=\|\sum_{n\in G}T(v_n)\|^2=\sum_{n\in G}\|T_n(v_n)\|^2\leq
 \sum_{n\in G}\varepsilon^2\|v_n\|^2\leq \varepsilon^2,$$
 as required.  
\end{proof}

\begin{definition}\label{def-almost}
We say that $T, S\in \bb$ almost commute if $[T, S]\in \K$.
If $\A\subseteq \bb$ is a C*-subalgebra, then $\B\subseteq \A$ is called an
almost masa of $\A$ if $\B$ is maximal (with respect to the inclusion)  C*-subalgebra of $\A$  consisting of elements
of $\A$ which almost pairwise commute.
\end{definition}

\begin{lemma}\label{sa-a-masa}
Suppose that $\A$ is a C*-subalgebra of $\bb$. For every almost commutative C*-subalgebra $\A'$ 
of $\A$ there is an almost masa $\B$ of $\A$ such that $\A'\subseteq\B$.
\end{lemma}
\begin{proof}
$\pi[\A']\subseteq \pi[\A]$ is a commutative C*-subalgebra of $\pi[\A]$.  So there is
a masa $\B'$ of $\pi[\A]$ which contains $\pi[\A']$. Let $\B=\pi^{-1}[\B']\cap\A$.  Then
$\A'\subseteq \B\subseteq \A$ and $\B$ is a C*-algebra because
 the image of a C*-algebra under a $*$-homomorphism is closed (II.5.1.2 of \cite{blackadar}).  Also
if $b, b'\in \B$, then $\pi(b), \pi(b')\in \B'$, so $\pi(b)\pi(b')-\pi(b')\pi(b)=\pi(bb'-b'b)=0$
and so $bb'-b'b\in \K$ as needed for $\B$ being almost commuting.
If $\B\subseteq \B''\subseteq \A$ is almost commuting C*-algebra, then $\pi[\A]\supseteq \pi[\B'']\supseteq \B'$
is commutative C*-subalgebra of $\pi[\A]$, so  $\pi[\B'']=\B'$, since $\B'$ is a masa of $\pi[\A]$.
It follows that $\B''\subseteq \pi^{-1}[\B']\cap \A=\B$ and so $\B''=\B$ which shows that $\B$ is almost masa of $\A$.
\end{proof}

\begin{lemma}\label{crucial} Suppose that $\A\subseteq \bb$ satisfies $\D_0\subseteq \A\subseteq \A_0$
and $\A$ is an almost masa of $\A_0$. Then $\pi[\A]$ is a masa of $\QQ$.
\end{lemma}
\begin{proof}  Note that since $\A$ is an almost masa of $\A_0$, $\pi[\A]$ is commutative,
and so by Lemma \ref{sa-masa} to prove that $\pi[\A]$ is a masa of $\QQ$ it is enough
to show that any self-adjoint element of $\QQ$ which commutes with the entire $\pi[\A]$
belongs to $\pi[\A]$.  So fix such a self-adjoint element. We may assume that it is  of the form
$\pi(S)$, where $S\in \bb$ is self-adjoint.  As $\pi(S)$ commutes with all elements of $\pi[\A]$ we have that
$S$ almost commutes with all elements of $\A$.
We will show that $S=V+W$, where $V\in \A$ and $W\in \K$, i.e., that $\pi(S)\in \pi[\A]$.
The idea of the proof is to transfer this setup to a setup where we can apply the theorem of Johnson and Parrott.

Consider Hilbert spaces $\ell_2(2\N)=\overline{span}(\{e_{2n}: n\in \N\})$ and
$\ell_2(2\N+1)=\overline{span}(\{e_{2n+1}: n\in \N\})$. We have $\ell_2=\ell_2(2\N)\oplus\ell_2(2\N+1)$.
Fix the induced $*$-isomorphism 
$$\eta: \bb\rightarrow \B(\ell_2(2\N)\oplus\ell_2(2\N+1)).$$
Note that there is a unitary isomorphism of Hilbert spaces $\ell_2(2\N)\oplus\ell_2(2\N+1)$ and $\ell_2\oplus \ell_2$
which sends $\{0\}\oplus\ell_2(2\N+1)$ onto  $\{0\}\oplus \ell_2$ and
$\ell_2(2\N)\oplus\{0\}$ onto $\ell_2\oplus \{0\}$ and that it induces a $*$-isomorphism 
$$\psi: \B(\ell_2(2\N)\oplus\ell_2(2\N+1))\rightarrow \B(\ell_2\oplus\ell_2).$$
As in \cite{murphy} pages 94-95 we see that there is a $*$-isomorphism 
$$\phi: \B(\ell_2\oplus\ell_2)\rightarrow M_2(\bb),$$
 where  $M_2(\bb)$ is the $C^*$-algebra of
two by two matrices with their entries in $\bb$ with the obvious operations and the norm inherited from $\B(\ell_2\oplus\ell_2)$
as in Theorem 3.4.2 of \cite{murphy}.
So given $R\in \bb$ we may represent it by $(\phi\circ\psi\circ\eta)(R)\in M_2(\bb)$ which is of the form
$$\begin{pmatrix}
R_{0, 0} & R_{0, 1}  \\
R_{1, 0} & R_{1, 1} 
\end{pmatrix},
$$
where, for $i, j\in \{0,1\}$ we have  $R_{i, j}\in\bb$  satisfying for each $k, m\in \N$
$$\langle R_{i, j}(e_k), e_m\rangle=\langle R(e_{2k+j}), e_{2m+i}\rangle.$$
Moreover,  $R$ is finite dimensional, if and only if   all $R_{i,j}$s for $i, j\in \{0,1\}$ are finite dimensional, so 
$$ \hbox{$R$ is compact, if and only if  all $R_{i,j}$s for $i, j\in \{0,1\}$ are compact.}\leqno (*)$$
 Note that 
$T\in \A_0$ if and only if $T_{i, j}$s are all diagonal for $i, j\in \{0,1\}$ and
$T\in \D_0$ if and only if $T_{0,0}=D=T_{1,1}$ is diagonal and $T_{0,1}=T_{1,0}=0$. By 
inspection we check that
$$\Bigg[\begin{pmatrix}
S_{0, 0} & S_{0, 1}  \\
S_{1, 0} & S_{1, 1} 
\end{pmatrix}, 
\begin{pmatrix}
D & 0  \\
0 & D
\end{pmatrix}\Bigg]=
\begin{pmatrix}
[S_{0, 0}, D] & [S_{0, 1}, D]  \\
[S_{1, 0}, D] & [S_{1, 1}, D] 
\end{pmatrix}.
$$
So if $S$ almost commutes with all $T\in \D_0\subseteq\A$, then $S_{i, j}$ almost commute with all diagonal operators
$D$ for  $i, j\in\{0,1\}$ by (*). As by the  Johnson and Parrott theorem the image of the algebra of diagonal operators under $\pi$
is a masa in $\QQ$. By Lemma \ref{sa-masa} this means that  $\pi(S_{i, j})$s are in this masa, which means that
$S_{i, j}=D_{i, j}+W_{i, j}$, where $D_{i, j}$ are diagonal and $W_{i, j}$ are compact for $i, j\in \{0,1\}$.
This, in turn by (*), means that $S=V+W$, where $V\in \A_0$ and $W\in \K$.  

Before going further we would like to have the above decomposition of $S$, where additionally
$V\in \A_0$ is self-adjoint. To obtain it we will use the choice of $S$ as a self-adjoint element of $\bb$.
From this we have $S=V^*+W^*$, so $S=(V+V^*)/2+(W+W^*)/2$. Here $(V+V^*)/2\in \A_0$ is self-adjoint and
$(W+W^*)/2\in \K$, so we can replace $V$ with $(V+V^*)/2$ and assume that $V\in \A_0$ is
self-adjoint indeed.

For any element $U\in \A$
$$[V, U]=[S-W, U]=[S, U]+W'=W''$$
for some $W', W''\in \K$ since $S$ almost commutes with all elements of $\A$. So
now as $V=V^*\in \A_0$ and $V=V^*$ almost commutes with the entire $\A$ and
$\A$ is an almost masa in $\A_0$, it follows that the C*-algebra generated by $\A$ and $V$
must be in fact $\A$, hence $V\in \A$. 
 So $S=V+W$, where $V\in \A$ and $W\in \K$ as required.
\end{proof}

The almost masas $\M_{\X, \mathfrak A}$ of $\A_0$ we will construct in the remainder of this section
will satisfy Lemma \ref{crucial}, i.e., they will satisfy $\mathcal D_0\subseteq \M_{\X, \mathfrak A}\subseteq \A_0$.
We are ready for more definitions leading to Definition \ref{def-M}.

\begin{definition}\label{def-Ds}
 For $\alpha\in[71/72, 1]$, $X\subseteq \N$ and $n\in \N$ we define 
 \begin{enumerate}
 \item  an orthonormal 
 basis $\F_{\alpha, n}=\{f_{\alpha, 2n}, f_{\alpha, 2n+1}\}$ of $span(\{e_{2n}, e_{2n+1}\})$, where
 \begin{enumerate}
 \item $f_{\alpha, 2n}=\alpha e_{2n}+ \sqrt{1-\alpha^2}e_{2n+1}$, 
 \item $f_{\alpha, 2n+1}=\sqrt{1-\alpha^2}e_{2n}-\alpha e_{2n+1}$.
 \end{enumerate}
 \item $F_{\alpha, n}\in \B_n$ is the projection on $\C f_{\alpha, 2n}$ and $F_{\alpha, n}^\perp\in \B_n$ its complementary
 projection in $\B_n$ onto $\C f_{\alpha, 2n+1}$. 
  \item $E_ n\in \B_n$ is the projection on $\C e_{2n}$ and $E_n^\perp\in \B_n$ its complementary
 projection in $\B_n$  onto $\C e_{2n+1}$. 
\item The algebra $\D_{\alpha, n}\subseteq \B_n$ 
of all elements of $\B_n$ diagonalizable with respect to
$\F_{\alpha, n}$.
\item The algebra $\D(X, \alpha)$
of all $T=(T_n)_{n\in \N}\in \A_0(X)$ such that
$T_n\in \D_{\alpha, n}$ for every $n\in X$.
\item The algebra $\D_\kk(X, \alpha)$
of all $T\in \A_0(X)$ such that $T=S+R$ for $S\in \D(X, \alpha)$ and
$R\in \K\cap \A_0(X)$.
\end{enumerate}
 \end{definition}
 
 Most often we will use $\alpha\in (71/72, 1]$. This is motivated by the calculations in the following lemma which will
 be exploited in the proof of Lemma \ref{gen-projections}:
 
 \begin{lemma}\label{71/72} If $\alpha\in(71/72, 1]$ and $n\in\N$, then $\|E_n-F_{\alpha, n}\|< 1/3$. 
 Consequently if $Q, Q^{\perp}$ are complementary one-dimensional projections
 in $\BB_n$, then one of the following conditions holds: 
 \begin{enumerate}
 \item $\|F_{\alpha, n}-Q\|, \|F_{\alpha, n}-Q^\perp\|, \|F_{\alpha, n}^\perp-Q\|, \|F_{\alpha, n}^\perp-Q^\perp\|\geq 1/6$ or
 \item $\|E_n-Q\|<1/2$ and $\|E_n-Q^\perp\|> 1/2$,
 \item $\|E_n-Q^\perp\|<1/2$ and $\|E_n-Q\|> 1/2$,
 \end{enumerate}
 \end{lemma}
 \begin{proof} For the first part of the lemma let $v\in \ell_2(\{2n, 2n+1\})$ be a unit vector. We  have
 $E_n(v)=\langle v, e_{2n}\rangle e_{2n}$ and 
 $F_{\alpha, n}(v)=\langle v, f_{\alpha, 2n}\rangle f_{\alpha, 2n}$.  So
 $$\|E_n(v)-F_{\alpha, n}(v)\|= \|\langle v, e_{2n}\rangle e_{2n}-\langle v, f_{\alpha, 2n}\rangle e_{2n}+
 \langle v, f_{\alpha, 2n}\rangle e_{2n}-\langle v, f_{\alpha, 2n}\rangle f_{\alpha, 2n}\|=$$
 $$=\|\langle v, e_{2n}-f_{\alpha, 2n}\rangle e_{2n} + \langle v, f_{\alpha, 2n}\rangle(e_{2n}-f_{\alpha, 2n})\|\leq$$
 $$\|v\|\|e_{2n}-f_{\alpha, 2n}\|\|e_{2n}\|+\|v\|\|f_{\alpha, 2n}\|\|e_{2n}-f_{\alpha, 2n}\|=2\|e_{2n}-f_{\alpha, 2n}\|=$$
 $$=2\sqrt{(1-\alpha)^2+(1-\alpha^2)}=2\sqrt{2-2\alpha}<2\sqrt{144/72-142/72}=2\sqrt{1/36}=1/3.$$
 
 For the second part  first
 note that
 condition (1)  is equivalent to the conjunction of only two inequalities
 $\|F_{\alpha, n}-Q\|, \|F_{\alpha, n}-Q^\perp\|\geq 1/6$ because
$\|F_{\alpha, n}^\perp-Q\|=\|(I_n-F_{\alpha, n})-(I_n-Q^\perp)\|=\|F_{\alpha, n}-Q^\perp\|$
and analogously $\|F_{\alpha, n}^\perp-Q^\perp\|=\|F_{\alpha, n}-Q\|$. 
Now assume that (1) fails. The first  case is that this is because
$\|F_{\alpha, n}-Q\|<1/6$. Then by the first part of the lemma $\|E_n-Q\|<1/3+1/6=1/2$, so we have (2) since
$\|Q-Q^\perp\|\geq 1$.
The second subcase when $\|F_{\alpha, n}-Q^\perp\|<1/6$ employs an analogous argument to conclude (3). 
 \end{proof}

 \begin{lemma}\label{dx} For $X\subseteq \N$ and $\alpha\in (71/72, 1]$
 the C*-algebra $\D_\kk(X, \alpha)$
 is almost masa of $\A_0(X)$
 \end{lemma}
 \begin{proof} First note that $\D_\kk(X, \alpha)$ is closed because it is of the form
 $\pi^{-1}[\pi[\D(X, \alpha)]]\cap \A_0(X)$ and 
 because the image of a C*-algebra under a $*$-homomorphism is closed (II.5.1.2 of \cite{blackadar}).
 So it is clear that it is a C*-subalgebra of $\A_0(X)$.
 
 Suppose that  $\D_\kk(X, \alpha)$
 is not almost masa in $\A_0(X)$ and we will aim at contradiction.
 So by Lemma \ref{sa-a-masa} there is a bigger than $\D_\kk(X, \alpha)$ almost commutative C*-subalgebra $\B$ of $\A_0(X)$.
 Take $S\in \B\setminus \D_\kk(X, \alpha)$. Then $S^*\in \B\setminus \D_\kk(X, \alpha)$ and
 both $S, S^*$ 
 almost commute with each other and with all elements of $\D_\kk(X, \alpha)$.
 So $\pi(S), \pi(S^*)\in \pi[\A_0(X)]$ commute with each other and with all elements of $\pi[\D(X, \alpha)]$.
 As $\D(X, \alpha)$ is an atomic masa in $\B(\ell_2(\{2n, 2n+1: n\in X\}))$,
 by Johnson and Parrott theorem $\pi[\D(X, \alpha)]$ is a masa in
 $\pi[\B(\ell_2(\{2n, 2n+1: n\in X\}))]\supseteq \pi[\A_0(X)]$. 
 So $\pi(S), \pi(S^*)\in \pi[\D(X, \alpha)]$ because otherwise $\pi[\D(X, \alpha)]$
 would not be maximal.
So there are $T\in \D(X, \alpha)$ 
 and $R\in  \K$ such that $S=T+R$. As $R=S-T$ we conclude that $R\in \A_0(X)$ as well. 
 So $S\in \D_\kk(X, \alpha)$ a contradiction with the choice of $S$.
 \end{proof}

 \begin{definition}\label{def-M} Suppose that $\X$ is a family of subsets of $\N$ and 
  $\mathfrak A=(\alpha_X: X\in \X)$ with $\alpha_X\in (71/72, 1]$
  for all $X\in \X$. Then we define
$$\M_{\X, \mathfrak A}=\bigcap_{X\in \X}\{T\in \A_0: P_{X}TP_{X}\in \D_\kk(X, \alpha_X)\}. $$
 \end{definition}

 \begin{lemma}\label{containsD0} Suppose that $\X$ is a family of subsets of $\N$ and 
  $\mathfrak A=(\alpha_X: X\in \X)$ with $\alpha_X\in (71/72, 1]$
  for all $X\in \X$. Then $\M_{\X, \mathfrak A}$ is a C*-subalgebra of $\A_0$ containing $\D_0$.
 \end{lemma}
 \begin{proof}
It is enough to prove that each set
$ \{T\in \A_0: P_{X}TP_{X}\in \D_\kk(X, \alpha_X)\}$
 is a C*-algebra containing $\D_0$. Perhaps less obvious is the fact that this set is closed.
 But it is equal to 
 $$\A_0(\N\setminus X)\oplus\big(\pi^{-1}[\pi[\D(X, \alpha_X)]]\cap \A_0(X)\big)$$ so it is closed
 because the image of a C*-algebra under a $*$-homomorphism is closed (II.5.1.2 of \cite{blackadar}).
 
 \end{proof}
 
 \begin{definition}\label{def-wide} A family $\X$ of subsets of $\N$ is called wide if
 for every infinite $X'\subseteq \N$ there is $X\in \X$ such that $X\cap X'$ is infinite.
 \end{definition}
 
 \begin{lemma}\label{wide} If $\X$ is a wide family of subsets of $\N$ and $(a_n)_{n\in \N}\in \ell_\infty$, then 
 $(a_n)_{n\in \N}\in c_0$ if and only if $(a_n1_X(n))_{n\in \N}\in c_0$ for every $X\in \X$.
 \end{lemma}
 \begin{proof}
 If $(a_n)_{n\in \N}\not\in c_0$, then there is $\varepsilon>0$ and an infinite $X'\subseteq \N$ such that
 $|a_n|>\varepsilon$ for all $n\in X'$. Since $\X$ is wide, there is $X\in \X$ such that $X\cap X'$ is infinite.
 Then $(a_n1_X(n))_{n\in \N}\not\in c_0$. 
 \end{proof}
 
  \begin{lemma}\label{M-comm} Suppose that $\X$ is a family of subsets of $\N$ and 
  $\mathfrak A=(\alpha_X: X\in \X)$ with $\alpha_X\in (71/72, 1]$
  for all $X\in \X$.
  If $\X$ is wide, then $\M_{\X, \mathfrak A}$ is almost commutative.
 \end{lemma}
 \begin{proof}

 Suppose that $T, S\in \M_{\X, \mathfrak A}$. By the definition of $\M_{\X, \mathfrak A}$ for every $X\in \X$
 there are $T(X), S(X)\in \D(X, \alpha_X)$ and $K(X), L(X)\in \K\cap \A_0(X)$
 such that  $T_n=T(X)_n+K(X)_n$ and $S_n=S(X)_n+L(X)_n$ for all $n\in X$.
 Then
 $$[T_n, S_n]=[T(X)_n+K(X)_n, S(X)_n+L(X)_n]=$$
 $$[K(X)_n, L(X)_n]+[K(X)_n, S(X)_n]+ [T(X)_n, L(X)_n] $$ for every $n\in X$ 
 because $[T(X)_n, S(X)_n]=0$ since $\D_{\alpha_X, n}$  is commutative.
 This means by Lemma \ref{limit-compact} that  $(\|[T_n, S_n]\|1_X(n))_{n\in \N}\in c_0$ since
 $(\|K(X)_n\|)_{n\in \N}$ and $(\|L(X)_n\|)_{n\in \N}\in c_0$ .
 Since $X\in\X$ was arbitrary, Lemma \ref{wide} implies that $(\|[T_n, S_n]\|)_{n\in \N}\in c_0$,
 which, by Lemma \ref{limit-compact} results in $[T, S]$ being compact, as required.

 \end{proof}
 
 \begin{definition}\label{def-coherent} Suppose that $\X$ is a family of subsets of $\N$. 
 A system 
 $$\mathfrak A=(\alpha_X: X\in \X)$$ 
 is called coherent if $\alpha_X\in (71/72, 1]$ for all $X\in \X$ and 
 for every $X, Y\in \X$ with $X\cap Y$ infinite we have $\alpha_X=\alpha_Y$.
 \end{definition}
 
 \begin{lemma}\label{Ds-inM}
  Suppose that $\X$ is a family of subsets of $\N$ and 
  $\mathfrak A=(\alpha_X: X\in \X)$ is coherent. 
  Then $\D_\kk(X, \alpha_X)\subseteq \M_{\X, \mathfrak A}$ for every $X\in \X$.
 \end{lemma}
 \begin{proof}
 Let $T\in \D_\kk(X, \alpha_X)$ and $Y\in \X$. Let $T=T'+S$, where $T'\in \D(X, \alpha_X)$
 and $S\in \K\cap \A_0(X)$. We need to prove that
 $P_YTP_Y\in \D_\kk(Y, \alpha_Y)$. By Lemma \ref{limit-compact} it
  is enough to prove that $P_YT'P_Y\in \D_\kk(Y, \alpha_Y)$.
 The first case is of finite $X\cap Y$. Then $P_YTP_Y$ is nonzero only on
 finitely many coordinates and so compact and hence belonging to $\D_\kk(Y, \alpha_Y)$.
 The second case is of infinite $X\cap Y$. 
 Then  $\alpha_X=\alpha_Y$ by Definition \ref{def-coherent},
 so $(P_YT'P_Y)_n$ for $n\in \N$ is either
 zero or diagonal with respect to $\F_{\alpha_X, n}=\F_{\alpha_Y, n}$ and so in  $\D_\kk(Y, \alpha_Y)$
 as required.
 \end{proof}
 
   \begin{lemma}\label{almost-masa}  Suppose that $\X$ is a wide family of subsets of $\N$ and 
  $\mathfrak A=(\alpha_X: X\in \X)$ is coherent. 
  $\M_{\X, \mathfrak A}$ is almost masa of $\A_0$.
 \end{lemma}
 \begin{proof}  $\M_{\X, \mathfrak A}$ is a C*-subalgebra of $\A_0$ by Lemma \ref{containsD0}. 
 All pairs of elements of $\M_{\X, \mathfrak A}$ almost commute by Lemma \ref{M-comm}.
 Suppose that $\M_{\X, \mathfrak A}$ is not an almost masa of $\A_0$ and let us aim at a contradiction. 
 By Lemma \ref{sa-a-masa} this hypothesis means that there is an almost commutative C*-subalgebra $\mathcal N$ of $\A_0$ 
 which properly includes $\M_{\X, \mathfrak A}$. Let $T\in \mathcal N\setminus \M_{\X, \mathfrak A}$. 
 Then $T^*\in \mathcal N\setminus \M_{\X, \mathfrak A}$
 and both $T$ and $T^*$ almost commute with each other and  with the entire $\M_{\X, \mathfrak A}$.
 
 As  $\D_0, \D_\kk(X, \alpha_X)\subseteq \M_{\X, \mathfrak A}$ by Lemmas \ref{containsD0} and \ref{Ds-inM}, and
 $P_X\in \D_0$ for all $X\in \X$, we conclude that, in particular
 $T$, $T^*$ and so $P_XTP_X$,  $P_XT^*P_X=(P_XTP_X)^*$ almost commute with each other and
 with all elements of  
 $\D_\kk(X, \alpha_X)$ for any $X\in \X$.
 By Lemma \ref{dx} the algebra $\D_\kk(X, \alpha_X)$  is almost masa of
 $\A_0(X)$ for every $X\in \X$ and so $P_XTP_X\in \D_\kk(X, \alpha_X)$ for every $X\in \X$,
 hence $T\in \M_{\X, \mathfrak A}$ contradicting the choice of $T$.
 \end{proof}

 The following two lemmas will prove that almost masas  $\M_{\X, \mathfrak A}$
 are generated by projections if $\X$ is wide and $\mathfrak A$ is coherent.

 \begin{lemma}\label{splitting} Suppose that $X\subseteq \N$  and $S=(S_n)_{n\in\N}, T=(T_n)_{n\in\N}$
 are in $\A_0(X)$ and $\theta>0$. Assume that $T_n=a_nQ_n+b_nQ_n^\perp$ 
as well as $S_n=a_n'F_n+b_n'F_n^\perp$
 and for every $n\in X$ we have:
 \begin{enumerate}
 \item $Q_n, F_n$ are  one-dimensional  orthogonal projections in $\B_n$,
 \item $\lim_{n\in X}\|S_n-T_n\|=0$,
 \item $|a_n-b_n|>\theta$,
 \item $\|Q_n-F_n\|<5/6$ or $\|Q_n-F_n^\perp\|<5/6$.
 \end{enumerate}
 Then $\lim_{n\in X}\|Q_n-R_n\|=0$, where
 $$ R_n=
  \begin{cases}
    F_n & \text{if $\|Q_n-F_n\|<5/6$,} \\
    F_n^\perp & \text{otherwise.}
  \end{cases}
 $$
 \end{lemma}
 \begin{proof}

 Put $Y=\{n\in X: \|Q_n-F_n\|< 5/6\}$ and $Z=
 \{n\in X: \|Q_n-F_n^\perp\|< 5/6\}$. We have $X=Y\cup Z$ by  (4).
Notice that $Y$ and $Z$ need not be disjoint. However, the argument below will
show that $Y\cap Z$ is finite (in fact $\lim_{n\in Y}\|Q_n-F_n\|=0$
and $\lim_{n\in Z}\|Q_n-F_n^\perp\|=0$), so the choice made in the definition of
$R_n$ on $Y\cap Z$ does not affect the conclusion of the lemma.

 By multiplying $T_n-S_n=a_nQ_n+b_nQ_n^\perp-a_n'F_n-b_n'F_n^\perp$ 
by $Q_n$ from the left and by $F_n$ from the right
 we obtain that $(a_n-a_n')Q_nF_n$'s converge to zero. However
 $$1=\|Q_n\|=\|Q_nF_n+Q_n(I_n-F_n)\|\leq \|Q_nF_n\|+\|Q_n(Q_n-F_n)\|\leq $$
 $$\leq\|Q_nF_n\|+\|Q_n\|\|Q_n-F_n\|=\|Q_nF_n\|+\|Q_n-F_n\|,$$
 so for $n\in Y$ we have $\|Q_nF_n\|\geq1/6$ by the definition of $Y$  and so 
 $$\lim_{n\in Y}(a_n-a_n')=0.\leqno (*)$$
  As $Q_n-F_n=F_n^\perp-Q_n^\perp$, similar arguments involving multiplication
 of $T_n-S_n$ from the left by $Q_n^\perp$
  and from the right by $F_n^\perp$ show that
$$\lim_{n\in Y}(b_n-b_n')=0.\leqno (**)$$

 Now writing $Q_n^\perp, F_n^\perp$ as $I_n-Q_n$, $I_n-F_n$ respectively
 we have 
 $$T_n-S_n=(a_n-b_n)Q_n-(a_n'-b_n')F_n+(b_n-b_n')I_n=$$
 $$=(a_n-b_n)(Q_n-F_n)- [(a_n'-b_n')-(a_n-b_n)]F_n+(b_n-b_n')I_n.$$
 So since $\|T_n-S_n\|$s   converge to zero for $n\in X$, by (2) and (*) and (**)
 we conclude that $Q_n-R_n=Q_n-F_n$'s converge to zero for $n\in Y$.
 
 The case of $n\in Z$ is similar:  By multiplying 
 $T_n-S_n=a_nQ_n+b_nQ_n^\perp-a_n'F_n-b_n'F_n^\perp$ by $Q_n$ from the left and $F_n^\perp$ from the right
 we obtain that $(a_n-b_n')Q_nF_n^\perp$ converge to zero. As before for $F_n$ we calculate that
 $1=\|Q_n\|\leq\|Q_nF_n^\perp\|+\|Q_n-F_n^\perp\|$ and so 
  by the definition of $Z$ we have that 
 $$\lim_{n\in Z}(a_n-b_n')=0.\leqno(+)$$
  As $Q_n-F_n^\perp=F_n-Q_n^\perp$, similar arguments  involving multiplication from the left by $Q_n^\perp$
  and from the right by $F_n$ show that
$$\lim_{n\in Z}(b_n-a_n')=0.\leqno(++)$$
Now writing $Q_n^\perp, F_n$ as $I_n-Q_n$, $I_n-F_n^\perp$ respectively
 we have 
 $$T_n-S_n=(a_n-b_n)Q_n-(b_n'-a_n')F_n^\perp+(b_n-a_n')I_n=$$
 $$=(a_n-b_n)(Q_n-F_n^\perp)- [(b_n'-a_n')-(a_n-b_n)]F_n^\perp+(b_n-a_n')I_n.$$
 So since $\|T_n-S_n\|$s   converge to zero for $n\in Z$, by (2) and (+) and (++)
 we conclude that $Q_n-F_n^\perp$ converge to zero for $n\in Z$
and so $Q_n-R_n$ converge to zero for $n\in Z\setminus Y$ as well.
 
 \end{proof}
 
   \begin{lemma}\label{gen-projections}   Suppose that $\X$ is a wide family of subsets of $\N$ and 
  $\mathfrak A=(\alpha_X: X\in \X)$ is coherent.  Then $\M_{\X, \mathfrak A}$ is generated by projections.
 \end{lemma}
 \begin{proof}
 Let $T=(T_n)_{n\in \N}\in \M_{\X, \mathfrak A}$ be self-adjoint, $\|T\|\leq M\in \R_{+}$ 
 and for $n\in \N$ let $\{a_n, b_n\}$ be the spectrum of $T_n$
 and $\{v_n, w_n\}$  be an orthonormal  basis of eigenvectors  corresponding to
 eigenvalues $a_n, b_n$. So possibly $a_n=b_n$
 or even $a_n=b_n=0$.  Let $Q_n$, $Q_n^\perp$ be orthogonal projections in $\BB_n$ 
 onto $v_n$ and $w_n$ respectively. So for each $n\in \N$ we have
 $$T_n=a_nQ_n+b_nQ_n^\perp.\leqno (*)$$
 Fix $\varepsilon>0$  with $\varepsilon<1/4$. 
 We will find finitely many  projections $P(m)\in \M_{\X, \mathfrak A}$ and constants $c_m\in \C$ for $m<3l\in \N$ 
 such that $\|T-\sum_{m<3l}c_mP{(m)}-L\|<\varepsilon$, where $L\in \K\cap \M_{\X, \mathfrak A}$ is
 self-adjoint.
 This is enough since any $T\in \M_{\X, \mathfrak A}$ is a linear combination of self-adjoint elements
 and  any compact self-adjoint $L\in \M_{\X, \mathfrak A}$ can be approximated by a finite linear combination of
 one dimensional projections in $\M_{\X, \mathfrak A}$ by Lemma \ref{limit-compact}.
 
 Since  $\|T\|\leq M$ and $T$ is self-adjoint, we have $a_n, b_n\in [-M, M]$. Let $\{J_m: m<l\}$ for some $l\in\N$
 be a partition of the interval $[-M, M]$ into sets of diameter smaller than $\varepsilon/2$
 and let $c_m, c_{m+l}, c_{m+2l}\in J_m$ for all $m<l$.
 
 For $m<l$ consider the following sets:
 \begin{itemize}
 \item $X_{m}=\{n\in \N: |a_n-b_n|<\varepsilon;\  {a_n+b_n\over 2}\in J_m\},$
 \item $Z=\{n\in \N\setminus \bigcup_{m< l}X_{m}: \|E_n-Q_n\|, \|E_n- Q_n^\perp\|\geq 1/2\},$
 \item $Y_{m+l}=\{n\in \N\setminus(Z\cup \bigcup_{m< l}X_{m}): a_n\in J_m\},$
 \item $Z_{m+2l}=\{n\in \N\setminus(Z\cup \bigcup_{m<l}X_{m}): b_n\in J_m\},$
 \end{itemize}
 
For  $m<l$ define projections $P(m)=(P(m)_n)_{n\in\N}$ by demanding
$$ P(m)_n=
  \begin{cases}
    I_n & \text{if $n\in X_m$} \\
    0 & \text{otherwise,}
  \end{cases}
 $$
 For  $l\leq m< 2l$
$$ P(m)_n=
  \begin{cases}
    Q_n & \text{if $n\in Y_{m}$} \\
    0 & \text{otherwise,}
  \end{cases}
 $$
 For  $2l\leq m< 3l$
$$ P(m)_n=
  \begin{cases}
    Q_n^\perp & \text{if $n\in Z_{m}$} \\
    0 & \text{otherwise,}
  \end{cases}
 $$
 Moreover define an operator $L=(L_n)_{n\in\N}$ by putting
 $$ L_n=
  \begin{cases}
    T_n & \text{if $n\in Z$} \\
    0 & \text{otherwise}
  \end{cases}
 $$

Let 
$$S=\sum_{m<3l}c_mP(m)+L.\leqno(**)$$
It is clear that $P(m)$s are projections and $L$ is self-adjoint as $T$ is. 
Now we show that $L$ is compact and in $\M_{\X, \mathfrak A}$. 

For this  it is enough to see that $Z$ is finite. Suppose that $Z$ is infinite and let us aim at a contradiction. 
$ T\in \M_{\X, \mathfrak A}$ and $P_Z\in \D_0\subseteq \M_{\X, \mathfrak A}$ by Lemma \ref{containsD0}, 
so we have $L=P_ZTP_Z\in \M_{\X, \mathfrak A}$. Since $\X$ is wide,
there is $X\in \X$ such that $Z\cap X$ is infinite. 
Now apply Lemma \ref{71/72}, where
clauses (2) and (3) fail for $Q=Q_n$ and $n\in Z\cap X$ by the definition of $Z$. So 
clause (1) holds that is 
$$\|F_{\alpha_X, n}-Q_n\|, \|F_{\alpha_X, n}-Q_n^\perp\|\geq 1/6\leqno(***)$$ for all
$n\in X\cap Z$.  Since $T\in \M_{\X, \mathfrak A}$ there is $R\in \D(X, \alpha_X)$ and $L'\in \K\cap \A_0(X)$ such that
$$a_nQ_n+b_nQ_n^\perp=T_n=R_n+L_n'$$
 for $n\in X$.  Subtracting $b_nI_n$ from both sides we get 
 $$(a_n-b_n)Q_n=V_n+L_n',$$
 where $V_n=R_n-b_nI_n\in \D_{\alpha_X,n}$. 
 Let $U_n:\ell_2(\{2n, 2n+1\})\rightarrow \ell_2(\{0, 1\})$ be unitary
 such that $U_n(e_{2n+i})=e_i$ for $i\in \{0,1\}$.
 
 By the compactness
 of the unit ball in $\B_0$ and of $[-M, M]$ we may find an infinite
 $Z'\subseteq Z\cap X$ such that for $n\in Z'$ we have that $U_nQ_nU_n^*$s converge to
  $Q'$, $U_nV_nU_n^*$s converge to $V'$,
 $a_n$s converge to $a'$ and $b_n$s converge to $b'$. Note that
 by
 Lemma \ref{limit-compact} $U_nL_n'U_n^*$s converge to zero, so we have
 $(a'-b')Q'=V'$.  But $|a_n-b_n|\geq \varepsilon/2$ for $n\in Z'$ because it is disjoint from each $X_m$ for $m<l$
 and $Q_n$s are one-dimensional projections, so $Q'$ is a one dimensional projection. Also the element
 $V_n\in \D_{\alpha_X,n}$, so $V'$ is diagonal with respect to 
 $\F_{\alpha_X, 0}$, so $Q'$ is one dimensional projection diagonal with respect to 
 $\F_{\alpha_X, 0}$.
 There are two such projections $F_{\alpha_X, 0}$ and $F_{\alpha_X, 0}^\perp$, but
 $U_nQ_nU_n^*$s converging to any of them contradicts (***). This completes the proof that
 $Z$ is finite and so $L$ is compact and $L\in \M_{\X, \mathfrak A}$.

Now we will show that $\|T-S\|\leq\varepsilon$.
We check it on each coordinate $n\in \N$.
If $n\in X_m$ for some $m<l$, then $n$ does not belong to any other set $Z, Y_{m'+l}, Z_{m'+2l}$ for $m'<l$. So 
the sum in (**) reduces to $(c_mP(m))_n=c_mI_n=c_m(Q_n+Q_n^\perp)$ and we have
$$\|T_n-S_n \|=\|a_nQ_n+b_nQ_n^\perp-c_m(Q_n+Q_n^\perp)\|\leq$$
$$\|a_nQ_n+b_nQ_n^\perp-((a_n+b_n)/2)(Q_n+Q_n^\perp)+((a_n+b_n)/2-c_m)(Q_n+Q_n^\perp)\|\leq$$
$$\|((a_n-b_n)/2)Q_n+((b_n-a_n)/2)Q_n^\perp\|+\varepsilon/2\leq$$
$$\leq |a_n-b_n|/2+ \varepsilon/2\leq\varepsilon,$$
since $|a_n-b_n|\leq\varepsilon$, $(a_n+b_n)/2\in J_m$ for $n\in X_m$
 and the diameter of $J_m$ is less than $\varepsilon/2$.

If $n\in Z$, then $n$ does not belong to any of the sets $X_m$, $Y_{m+l}$, $Z_{m+2l}$ for $m<l$
so the sum (**) reduces to $L_n=T_n$, so
$T_n-S_n=T_n-T_n=0$. 

If $n\not \in Z\cup \bigcup_{m< l}X_{m}$, then 
there are $l\leq m<2l, 2l\leq m'<3l$ such that $n\in Y_m\cap Z_{m'}$. Then $S_n=c_mQ_n+c_{m'}Q_n^\perp$,
so 
$$\|T_n-S_n\|=\|a_nQ_n+b_nQ_n^\perp-c_mQ_n-c_{m'}Q_n^\perp\|
\leq |a_n-c_m|+|b_n-c_{m'}|<\varepsilon/2+\varepsilon/2=
\varepsilon$$
because the diameters of $J_{m''}$s  for $m''<l$ are less than $\varepsilon/2$.

So we are left with proving that  all the projections $P(m)$ for $m<3l$ are in $\M_{\X, \mathfrak A}$.
 Let $X\in \X$. We need to show that $P_XP(m)P_X\in \D_\kk(X, \alpha_X)$.
 For $m<l$ it is clear because then $P(m)\in \D_0\subseteq \M_{\X, \mathfrak A}$
 by Lemma \ref{containsD0}. So consider $l\leq m<2l$.  For $n\in \N$ let
$$ R_n=
  \begin{cases}
    F_{\alpha_X, n} & \text{if $n\in Y_m\cap X$ and $\|E_n-Q_n\|<1/2$} \\
    F_{\alpha_X, n}^\perp & \text{if $n\in Y_m\cap X$ and $\|E_n-Q_n^\perp\|<1/2$} \\
    0 & \text{otherwise.}
  \end{cases}
 $$
 $R_n$ is well defined because $\|Q_n-Q_n^\perp\|=1$ by a functional calculus argument and so only one
 of the conditions $\|E_n-Q_n\|<1/2$ or $\|E_n-Q_n^\perp\|<1/2$ can be true for a given $n\in \N$.
 It is also  clear that $R=(R_n)_{n\in \N}\in\D(X, \alpha_X)$.
 It is enough to prove that $P_X(P(m)-R)P_X$ is compact and for this, 
 by Lemma \ref{limit-compact}  we need to show that 
 $\lim_{n\in X}\|(P(m)-R)_n\|=0$. 
 First, note that since $T\in \M_{\X, \mathfrak A}$, 
 there is  $S\in \D(X, \alpha_X)$ such that $\lim_{n\in X\cap Y_m}\|S_n-T_n\|=0$.
 By considering $(S+S^*)/2$ in place of $S$ and recalling that $T$ is self-adjoint,
  we may assume that $S$ is self-adjoint. Since it is in 
 $\D(X, \alpha_X)$ we have $S_n=a_n'F_{\alpha_X, n}+b_n'F_{\alpha_X, n}^\perp$. 
 By its definition  $P(m)_n$ is equal to $Q_n$ for $n\in Y_m$ and otherwise $0$
 
Note that if $\|E_n-Q_n\|<1/2$, then by Lemma \ref{71/72} we have
$\|Q_n-F_{\alpha_X, n}\|\leq\|Q_n-E_n\|+\|E_n-F_{\alpha_X, n}\|<1/2+1/3=5/6$
and similarly if $\|E_n-Q_n^\perp\|<1/2$, then $\|Q_n-F_{\alpha_X, n}^\perp\|=\|F_{\alpha_X, n}-Q_n^\perp\|<5/6$.
Moreover note that one of these cases takes place for each $n\in X\cap Y_m$ because $Y_m\cap Z=\emptyset$.
Moreover, by the definitions of the sets $X_r$ and $Y_m$, $|a_n-b_n|\geq\varepsilon>\varepsilon/2$
for every $n\in Y_m$. So $(S_n)_{n\in X\cap Y_m}$, $(T_n)_{n\in X\cap Y_m}$,
$(Q_n)_{n\in X\cap Y_m}$ and $(F_{\alpha_X,n})_{n\in X\cap Y_m}$
satisfy the hypothesis of Lemma \ref{splitting} with $\theta=\varepsilon/2$.
Let $(R_n')_{n\in X\cap Y_m}$ be the projections defined in Lemma \ref{splitting} for these sequences.
Then $\lim_{n\in X\cap Y_m}\|R_n'-Q_n\|=0$.
By the definition of $R_n$ above, $R_n$ can differ from $R_n'$ only when both
inequalities in (4) of Lemma \ref{splitting} hold. By the proof of that lemma
this happens for only finitely many $n\in X\cap Y_m$. Hence
$\lim_{n\in X\cap Y_m}\|R_n-Q_n\|=0$ and hence $P_XP(m)P_X\in \D_\kk(X, \alpha_X)$
as required for $P(m)\in\M_{\X, \mathfrak A}$ since $X$ was an arbitrary element of $\X$.

The proof that $P(m)\in \M_{\X, \mathfrak A}$ for $2l\leq m<3l$ is analogous.
 \end{proof}
 
 \begin{proposition}\label{prop-zfc}
Suppose that $\X$ is a wide family of subsets of $\N$ and $\mathfrak A=(\alpha_X: X\in \X)$ is
coherent. 
Then $\pi[\M_{\X, \mathfrak A}]$ is a masa of $\QQ$ which is generated by projections. 
 \end{proposition}
 \begin{proof} Apply Lemmas \ref{crucial}, \ref{containsD0}, \ref{almost-masa},  \ref{gen-projections}.
 \end{proof}
 
 In the remainder of this section we will look at particular cases of masas
 $\pi[\M_{\X, \mathfrak A}]$. First we consider $\X$ to be a maximal almost disjoint family of subsets
 of $\N$ because  $\pi[\M_{\X, \mathfrak A}]$ in this case  has rather simple Gelfand space 
 (Proposition \ref{few-projections} (3)).
 
 \begin{definition}\label{def-h-omega} A Boolean algebra $\A$ satisfies condition $H_\omega'$ whenever
 given $A_n, B_n\in \A$ satisfying
 $$A_n< A_{n+1}< \dots < B_{n+1}< B_n$$
 we can find $C\in \A$ satisfying $A_n< C<B_n$ for all $n\in \N$.
 \end{definition}
 
 \begin{lemma}\label{h-omega} If a Boolean algebra is complete or equal to $\wp(\N)/Fin$, then it satisfies
 condition $H_\omega'$
 \end{lemma}
 \begin{proof} For a complete Boolean algebra consider $C=\sup_{n\in \N}A_n$ and
 for $\wp(\N)/Fin$ use Lemma 1.1.2 of \cite{handbook-van-mill} where a stronger condition $H_\omega$
 is considered which allows one of the sequences $(A_n)_{n\in \N}$ or $(B_n)_{n\in \N}$ to be constant. 
 \end{proof}
 
 Recall that a 
   SAW*-algebra is a C*-algebra where for any two orthogonal positive elements $a, b$ there
   is positive $c$ such that $ac=a$ and $bc=0$ (p.16 of \cite{saw}). Commutative unital SAW*-algebras
   are known to be $*$-isomorphic to $C(K)$s where $K$ is a compact $F$-space, i.e.,
   where any two disjoint open $F_\sigma$ subsets of $K$ have disjoint closures  (p.16 of \cite{saw}).
 
 \begin{lemma}\label{f-space} If $K$ is totally disconnected compact $F$-space (i.e., $C(K)$
 is a SAW*-algebra), then the Boolean algebra $Clop(K)$ of clopen subsets of $K$
 satisfies $H_\omega'$ of Definition \ref{def-h-omega}.
 \end{lemma}
 \begin{proof}
   The proof is the same as Lemma 1.2.3 of \cite{handbook-van-mill}.
  \end{proof}

 \begin{definition}\label{def-BX} Suppose that $\X$ is a maximal infinite almost disjoint family of infinite subsets of $\N$.
 Let $2\N$, $2\N+1$ denote the sets of all even and all odd elements of $\N$ respectively. We define:
 \begin{enumerate}
 \item  For $Y\subseteq \N:$ \ $\partial Y=\{n\in \N: |Y\cap\{2n, 2n+1\}|=1\}$
 \item 
 $\B_\X=\{Y\subseteq \N: 
\partial Y\subseteq X_1\cup \dots\cup X_k\cup F \ \hbox{for some}\  X_1, \dots X_k\in \X; k\in \N; F\subseteq \N 
\ \hbox{is finite}  \}$.
 \end{enumerate}
 \end{definition}
 As $\partial(\N\setminus Y)=\partial Y$, $\partial(Y\cup Y'), \partial(Y\cap Y')\subseteq \partial Y\cup \partial Y'$
 for every $Y, Y'\subseteq \N$, it is clear that $\B_\X$ is a Boolean subalgebra of $\wp(\N)$. Moreover 
 $$Fin, \wp\big(\bigcup_{n\in X_1\cup\dots\cup X_k}\{2n,2n+1\}\big)\subseteq \B_\X$$
for every $X_1, \dots, X_k\in \X$ and $k\in \N$.

 \begin{proposition}\label{few-projections} Suppose that $\X$ is an infinite maximal almost disjoint family and 
 $\mathfrak A=(\alpha_X: X\in \X)$ satisfies $\alpha_X\in (71/72, 1]$
 for $X\in \X$ and $\alpha_X\not=\alpha_Y$, for $X, Y\in \X$
 whenever $X\not=Y$. Then 
 \begin{enumerate}
 \item $\pi[\M_{\X, \mathfrak A}]$ is a  masa in  $\QQ$ generated by projections,
 \item  whenever $Q=(Q_n)_{n\in \N}\in \A_0$ is a projection
 and $\nabla Q=\{n: rank(Q_n)=1\}$ is not covered by the union of finitely many elements of $\X$ and a finite set,
 then $Q\not\in \M_{\X, \mathfrak A}$.
 \item $\pi[\M_{\X, \mathfrak A}]$ is $*$-isomorphic to $C(K_{\B_\X/Fin})$, where
 $K_{\B_\X/Fin}$ is the Stone space of the Boolean algebra $\B_\X/Fin$.
 \item $Proj(\pi[\M_{\X, \mathfrak A}])$ does not  satisfy  condition $H_\omega'$, and
 so $\pi[\M_{\X, \mathfrak A}]$ is not a SAW*-algebra and it does not have a commutative lift to $\bb$.
 \item $\pi[\M_{\X, \mathfrak A}]$ does not have the Grothendieck property, and so
 it does not admit a conditional expectation from $\QQ$.
 \end{enumerate}
 \end{proposition}
 \begin{proof} For (1) note that a maximal almost disjoint family is always a wide family of sets.
 Also $\mathfrak A$ satisfying the hypothesis of the proposition is trivially coherent.
 So by Proposition \ref{prop-zfc} the algebra  $\pi[\M_{\X, \mathfrak A}]$ is a masa of $\QQ$ generated by 
 projections.
 
 Let $U_n:\ell_2(\{2n, 2n+1\})\rightarrow \ell_2(\{0, 1\})$ be unitary
 such that $U_n(e_{2n+i})=e_i$ for $i\in \{0,1\}$.
 
 For the proof of (2) we will use the fact that if $X\in \X$, $Q=(Q_n)_{n\in \N}\in \M_{\X, \mathfrak A}$ and 
 $Q_n$ is one-dimensional projection for all
 $n\in X'$ for some infinite $X'\subseteq X$, then 
 $$\lim_{n\in X'}\min(\|Q_n-F_{\alpha_X, n}\|, \|Q_n-F_{\alpha_X, n}^\perp\|)=0\leqno (*)$$
 This is because   by Lemma \ref{limit-compact}
  and Definition \ref{def-M} the only
 accumulation points of $U_nQ_nU^*_n\in \BB_0$ must be diagonalizable
  with respect to $\F_{\alpha_X, 0}$, but their
 accumulation points must also be one dimensional projections. However, there are  only two
 projections $F_{\alpha_X, 0}$ and $F_{\alpha_X, 0}^\perp$ which satisfy both of these requirements.

 Assume that $Q$ and $\nabla Q$ are as in (2).
 By the maximality of $\X$ we can recursively build distinct $X_k\in \X$ such that
 $X_k\cap \nabla Q$ is infinite for all $k\in \N$. By passing to a subsequence we may assume that
 $(\alpha_{X_k})_{k\in \N}$ converges to $\beta\in [71/72,1]$.  
 Using (*) for any $k\in \N$ find finite $G_k\subseteq X_k$ such that 
 $$\|Q_n-F_{\alpha_{X_k}, n}\|<1/4(k+1)\ \hbox{or}\ \|Q_n-F_{\alpha_{X_k}, n}^\perp\|<1/4(k+1)\leqno(**)$$
 for all $n\in X_k\cap \nabla Q\setminus G_k$ and any $k\in \N$ and $(X_k\setminus G_k)\cap(X_{k'}\setminus G_{k'})=\emptyset$
 for any two distinct $k, k'\in\N$.
 
 Let $Z=\bigcup_{k\in \N}((X_k\cap \nabla Q)\setminus G_k)$. 
 Note that $Z$ is not covered by the union of finitely many elements of almost disjoint $\X$ and a finite set
 because it intersects infinitely many of them on infinite sets.
 Considering the family of all infinite sets of the form $Z\cap X$ for  $X\in \X$, 
 using the maximality of $\X$ we see that this family is
 a maximal infinite almost disjoint family of infinite subsets of $Z$ and such families cannot be countable. 
 So  there are uncountably many $X\in \X\setminus \{X_k: k\in \N\}$ with $X\cap Z$ infinite.  
We may find $X\in \X\setminus \{X_k: k\in \N\}$ such that $X\cap Z$ is infinite
 and $\alpha_X\not=\beta$. The set $X\cap Z$ must intersect infinitely many $X_k$s, because 
 $X\cap Z$ is infinite and $X\cap X_k$ are finite since $X, X_k\in \X$ are distinct. So
 there is an  infinite $A\subseteq \N$ such that
 for $k\in A$  there are $n_k\in X\cap(X_k\cap \nabla Q\setminus G_k)$ which  by
 (**) means that
 $$\|Q_{n_k}-F_{\alpha_{X_k}, n_k}\|<1/4(k+1)\ \hbox{or}\ \|Q_{n_k}-F_{\alpha_{X_k}, n_k}^\perp\|<1/4(k+1)$$
 But (*) also implies that 
 $$\lim_{k\in A}\min(\|Q_{n_k}-F_{\alpha_X, n_k}\|, \|Q_{n_k}-F_{\alpha_X, n_k}^\perp\|)=0.$$
 The above two statements imply that
 $$\lim_{k\in A}\min(\{\|S-T\|: S\in \{F_{\alpha_X, n_k}, F_{\alpha_X, n_k}^\perp\}, 
 T\in\{F_{\alpha_{X_k}, n_k}, F_{\alpha_{X_k}, n_k}^\perp\} \})=0.$$
 Applying the conjugation by $U_{n_k}$ this gives that 
 $\{F_{\beta, 0}, F_{\beta, 0}^\perp\}\cap \{F_{\alpha_X, 0}, F_{\alpha_X, 0}^\perp\}\not=\emptyset$
 since $\alpha_{X_k}$s converge to $\beta$.
But this is impossible for distinct $\alpha_X, \beta\in [71/72, 1]$, which completes the proof of (2).
 
 For (3) 
 first, given a finite $\Y=\{X_1,  \dots, X_k\}\subseteq \X$ and finite $F\subseteq \N$ 
 define a Boolean algebra 
 $$\B_{\Y, F}
=\{Y\subseteq \N: 
\partial Y\subseteq \bigcup\Y\cup F  \}\subseteq \B_\X.$$ 
For $Z\in \B_{\Y, F}$ define $h^{\Y, F}(Z)\in Proj(\M_{\X, \mathfrak A})$ by
  $$ h^{\Y, F}(Z)_n=
  \begin{cases}
    F_{\alpha_{X_i}, n} & \text{if $2n\in Z$   and $n\in \partial Z\cap X_i\setminus
    \big(\bigcup_{j\not= i}X_j\cup F\big)$ and $1\leq i\leq k$,} \\
    F_{\alpha_{X_i}, n}^\perp & \text{if $2n+1\in Z$ and $n\in \partial Z\cap  
    X_i\setminus\big(\bigcup_{j\not= i}X_j\cup F\big)$ and $1\leq i\leq k$,} \\
    E_{n} &  \text{if $2n\in Z$   and $n\in \partial Z\cap \big(F\cup \bigcup\{X_i\cap X_j: 1\leq i<j\leq k\}\big)$,} \\
    E_{n}^\perp &  \text{if $2n+1\in Z$   and $n\in \partial Z\cap \big(F\cup \bigcup\{X_i\cap X_j: 1\leq i<j\leq k\}\big)$,} \\    
    I_n & \text{if $\{2n, 2n+1\}\subseteq Z$,} \\
    0_n & \text{if $\{2n, 2n+1\}\cap Z=\emptyset$.}
  \end{cases}
 $$
 Since distinct cases in the above definition are excluding each other, it should be clear that
 $h^{\Y, F}:\B_{\Y, F} \rightarrow Proj(\M_{\X, \mathfrak A})$ is a Boolean monomorphism.
 Also $h^{\Y, F}(Z)$ and $h^{\Y', F'}(Z')$ differ by a compact operator if
 $Z=^*Z'$ and  $Z\in \B_{\Y, F}$ and $Z'\in \B_{\Y', F'}$ because $\B_{\Y, F},  \B_{\Y', F'}\subseteq 
 \B_{\Y\cup\Y', F\cup F'}$.
 It follows that $h: \rho[\B_\X]\rightarrow \pi[Proj(\M_{\X, \mathfrak A})]$ defined by
 $$h(\rho(Z))=\pi(h^{\Y, F}(Z)),$$
 where $\Y\subseteq\X$ and $F\subseteq \N$ are finite such that $Z\in \B_{\Y, F}$ is a Boolean homomorphism. 
 Moreover it is a monomorphism since if $Z\not\in Fin$, then $h^{\Y, F}(Z)$ is noncompact by Lemma \ref{limit-compact}.
 Since $Fin\subseteq \B_\X$ is
 the kernel of $\rho$ we conclude that $\rho[\B_\X]=\B_\X/Fin$. 
 So to prove (3) we are left with proving that $h$ is onto $Proj(\pi[\M_{\X, \mathfrak A}])$.
 
 It is enough to prove that $h$ is onto $\pi[Proj(\M_{\X, \mathfrak A})]$. This is because
 $\M_{\X, \mathfrak A}$ is generated by its projections by Lemma \ref{gen-projections}, so
  $\pi[\M_{\X, \mathfrak A}]$ is generated by 
  $\pi[Proj(\M_{\X, \mathfrak A})]$. However, if a commutative C*-algebra $C(K)$ is generated by
  some Boolean algebra $\CC$ of its projections, then $\CC$ contains all projections (any
  clopen set can be written as a finite union of finite intersections of clopen sets corresponding to elements of $\CC$, because
  of the compactness and separation of points of $K$ by $\CC$).
 
To prove that $h$ is onto $\pi[Proj(\M_{\X, \mathfrak A})]$ we use (2).
For any $Q\in Proj(\M_{\X, \mathfrak A})$ there are finite
$\Y\subseteq\X$ and finite $F\subseteq\N$ such that
$\nabla Q\subseteq\bigcup\Y\cup F$.
By the definition of $\M_{\X,\mathfrak A}$ and Lemma \ref{limit-compact},
there is $Z\in\B_{\Y,F}$ such that
$$Q-h^{\Y,F}(Z)\in\K.$$
Consequently
$$\pi(Q)=\pi(h^{\Y,F}(Z))=h(\rho(Z)),$$
and hence $\pi(Q)\in h[\rho[\B_\X]]$, as required.

To prove (4)  for  $k\in \N$ pick distinct 
$X_k\in\X$
  and  finite $F_k\subseteq X_k$ such that $(X_k\setminus F_k)$s are pairwise disjoint.
  
For $m\in \N$ define projections $P(m), Q(m)$ 
  $$ P(m)_n=
  \begin{cases}
    F_{\alpha_{X_k}, n} & \text{for $n\in X_k\setminus F_k$, \ $k\leq m$} \\
    0_n & \text{otherwise.}
  \end{cases}
 $$
 $$ Q(m)_n=
  \begin{cases}
    F_{\alpha_{X_k}, n} & \text{for $n\in X_k\setminus F_k$, \ $k\leq m$} \\
    I_n & \text{otherwise.}
  \end{cases}
 $$
 Note that $P(m), Q(m)\in\M_{\X, {\mathfrak A}}$ by Lemma \ref{Ds-inM}
  and  
  $$\pi(P(m))<\pi(P(m+1))< \dots <\pi(Q(m+1))<\pi(Q(m)).$$
  To prove that condition $H_\omega'$ fails
 in $Proj(\pi[\M_{\X, {\mathfrak A}}])$ (recalling from the proof of (3) that
 $Proj(\pi[\M_{\X, {\mathfrak A}}])=\pi[Proj(\M_{\X, {\mathfrak A}})]$) 
  it is enough to show that
 any projection $R\in \A_0$ satisfying $\pi(P(m))<\pi(R)<\pi(Q(m))$ for every $m\in \N$ 
 satisfies $X_k\subseteq^*\nabla R$  for every $k\in \N$ which is impossible
 inside $\M_{\X, {\mathfrak A}}$ by (2). 
 
 If $R=(R_n)_{n\in \N}\in \A_0$ is a projection satisfying 
 $\pi(P(m))\leq\pi(R)\leq\pi(Q(m))$, then $(P(m)_nR_n-P(m)_n)_{n\in \N}$
 and $(Q(m)_nR_n-R_n)_{n\in \N}$ converge to zero by Lemma \ref{limit-compact}.
 Since both $P(m)_n$ and $Q(m)_n$ are one-dimensional projections
 for almost all $n\in X_k$ for $k\leq m$, it follows that
 $R_n$ cannot be $0$ on infinitely many $n\in X_k$ for each $k$
 and $R_n$ cannot be $I_n$ on infinitely many $n\in X_k$ for each $k$.
 It follows that $R_n$ is one dimensional projection 
 for almost all $n\in X_k$ for all $k\in \N$ as required for (4).

To conclude the second part of (4) recall that by Lemma \ref{masa-masaQQ} 
masas of $\QQ$ that have commutative lifts are of the form $C(K)$, where
$K$ is one of the spaces listed in items (1) - (3) in the Introduction.
The Boolean algebra of clopen subsets of such $K$s are complete or isomorphic to
$\wp(\N)/Fin$, or direct sums of such algebras, so they satisfy condition $H_\omega'$ by Lemma \ref{h-omega}. 
By Lemma \ref{f-space}, this also implies that $\pi[\M_{\X, \mathfrak A}]$ is not a SAW*-algebra.

To prove (5) recall that a Banach space $B$ has the Grothendieck property if
the weak$^*$ convergent sequences in the dual $B^*$ of $B$ are weakly convergent.
For more on the Grothendieck property see \cite{kania}. 
So to prove (5) we need to produce a sequence of elements of $(\pi[\M_{\X, \mathfrak A}])^*$.
Let $X_k\in \X$ be distinct for $k\in \N$ and let $u_k$ be a fixed nonprincipal ultrafilter in $\wp(\N)$
which contains $X_k$.  Recall that if $(c_n)_{n\in N}$ is a sequence of complex numbers,
then the limit $\lim_{u}c_n$ of $(c_n)_{n\in N}$ along $u$ is a complex number $c$ such that
for each $\varepsilon>0$ there is $A\in u$ such that $|c_n-c|<\varepsilon$ for all $n\in A$.
It is clear that for each bounded sequence $(c_n)_{n\in \N}$ there is $\lim_{u}c_n$.

When $T\in \M_{\X, \mathfrak A}$ and $X\in \X$, by Definition \ref{def-M} we have
$P_XTP_X\in \D_\kk(X, \alpha_X)$ which means that 
$P_XTP_X=D(T, X)+K(T, X)$, where
$D(T, X)_n$ is diagonal with respect to $\F_{\alpha_X, n}$ for $n\in X$ and $K(T, X)\in \K\cap \A_0(X)$
by Definition \ref{def-Ds}. So 
$$D(T, X)_n=a_n(D(T, X))F_{\alpha_X, n}+b_n(D(T, X))F_{\alpha_X, n}^\perp.$$
Define $x_{k, i}: \M_{\X, \mathfrak A}\rightarrow \C$ for $k\in \N$ and $i=0,1$ as 
$$x_{k, 0}(T)=\lim_{u_k}a_n(D(T, X_k)), \ x_{k, 1}(T)=\lim_{u_k}b_n(D(T, X_k)).$$
Of course the values of $a_n(D(T, X))$s and $b_n(D(T, X))$s depend on the choice of $D(T, X)$ and $K(T, X)$
which is not unique, but for any other such choice $D'(T, X)$ and $K'(T, X)$ satisfying
$D'(T, X)+K'(T, X)=P_XTP_X=D(T, X)+K(T, X)$ we have $D'(T, X)-D(T, X)$ compact
so $(a_n(D(T, X))-a_n(D'(T, X)))_{n\in \N}$ and $(b_n(D(T, X))-b_n(D'(T, X)))_{n\in \N}$ converge to zero
by Lemma \ref{limit-compact}. So as $u_k$s are nonprincipal, the values of
$x_{k, i}$ for $k\in \N$ and $i=0,1$ actually do not depend on the choice of
$D(T, X)$ and $K(T, X)$ and so are well defined for $T\in \M_{\X, \mathfrak A}$.

The definition of $x_{k, i}$ for $k\in \N$ and $i=0,1$ easily implies that 
they are multiplicative linear functionals of norm one on  $\M_{\X, \mathfrak A}$ which are
null on compact operators, so they define  by $y_{k, i}(\pi(T))=x_{k, i}(T)$ multiplicative linear
 functionals $y_{k, i}$ for $k\in \N$ and $i=0,1$ on  $\pi[\M_{\X, \mathfrak A}]$.
Moreover $y_{k,0}\ne y_{k,1}$ for every $k\in\N$. Indeed, let $R^k\in \D(X_k,\alpha_{X_k})$
be the projection such that $R^k_n=F_{\alpha_{X_k},n}$ for $n\in X_k$ and $R^k_n=0$
otherwise. By Lemma \ref{Ds-inM}, $R^k\in\M_{\X,\mathfrak A}$, while
$y_{k,0}(\pi(R^k))=1$ and $y_{k,1}(\pi(R^k))=0$.

 Hence $y_{k, i}$ can be associated with Dirac measures $\delta_{t_{k, i}}$ 
concentrated in points $t_{k, i}$ of the Gelfand space of 
 $\pi[\M_{\X, \mathfrak A}]$ which belong to pairwise disjoint clopen sets corresponding
 to the projections $\pi(P_{X_k})$. It follows that  the sequence of 
 $$z_k=y_{k, 0}-y_{k, 1}$$
 does not weakly converge in $(\pi[\M_{\X, \mathfrak A}])^*$. This is because
 $G=\{t_{2k, 0}: k\in \N\}$ is a Borel subset of the Gelfand space of $\pi[\M_{\X, \mathfrak A}]$ since it is countable,
 so $G$ defines a functional $\phi_G$ in the bidual $(\pi[\M_{\X, \mathfrak A}])^{**}$  by taking
 values of  Radon measures (associated with elements of the dual of $\pi[\M_{\X, \mathfrak A}]$ by the Riesz representation
 theorem for the duals of $C(K)$ spaces) on $G$. Moreover 
  $$ \phi_G(z_k)=z_k(G)=\delta_{t_{k, 0}}(G)-\delta_{t_{k, 1}}(G)=
  \begin{cases}
    1-0=1 & \text{for $k\in \N$ even} \\
    0-0=0 & \text{for $k\in \N$ odd}
  \end{cases}
 $$
 So $\phi_G$ is a witness of the fact that $(z_k)_{k\in \N}$ does not converge in the weak topology on the
 dual $(\pi[\M_{\X, \mathfrak A}])^{*}$.
 
 Now we will show that $(z_k)_{k\in \N}$ does  converge to zero in the weak$^*$ topology on the
 dual $(\pi[\M_{\X, \mathfrak A}])^{*}$ which will complete the proof of the failure of the Grothendieck
 property for $\pi[\M_{\X, \mathfrak A}]$. Since $\M_{\X, \mathfrak A}$ is generated by 
 projections, it is enough to show that for any projection $Q\in \M_{\X, \mathfrak A}$
 we have  $z_k(\pi(Q))=0$ for all but finitely many $k\in \N$. 
 
 Let $Q=(Q_n)_{n\in \N}$ be a projection in $\M_{\X, \mathfrak A}$. 
By (2) for all but finitely many $k$s the projections $Q_n$ are either $0$ or $I_n$
for almost all $n\in X_k$. We may therefore choose $D(Q,X_k)$ so that
$D(Q,X_k)_n=Q_n$ for almost all $n\in X_k$. Hence
$a_n(D(Q,X_k))=b_n(D(Q,X_k))$ for almost all $n\in X_k$, and so
$$z_k(\pi(Q))=y_{k,0}(\pi(Q))-y_{k,1}(\pi(Q))
=x_{k,0}(Q)-x_{k,1}(Q)=0$$
 as needed for the convergence of $z_k$s to zero in the weak$^*$ topology 
 in the dual $(\pi[\M_{\X, \mathfrak A}])^{*}$ since $z_k$ form a norm bounded sequence.
 
 To conclude the second part of (5) recall that the Grothendieck property is preserved by
 quotients of Banach spaces (Proposition 3.1.4 of \cite{kania}) and $\bb$ has the Grothendieck property
  by \cite{pfitzner} and so
 $\QQ$ has it as well. So since $\pi[\M_{\X, \mathfrak A}]$ fails the Grothendieck property,
 it is not isomorphic as a Banach space to a quotient of $\QQ$,  and so it cannot admit
 a conditional expectation from $\QQ$.
 \end{proof}
 
 Definitely,  one could use a simpler, diagonal argument to prove  the last part of Proposition \ref{few-projections} (4)
 that $\pi[\M_{\X, \mathfrak A}]$   does not have a commutative lift to $\bb$.
 Namely,  one takes a maximal almost disjoint family
 of cardinality $\cc$ and  chooses $\mathfrak A(X)$s for $X\in \X$ in $(71/72, 1]$ to disagree with the image under $\pi$
 of each masa in $\bb$  as there are only $\cc$ such masas by Proposition 12.3.1  of \cite{ilijas-book}.
 However, our argument shows a concrete property  which is the reason behind the non-lifting. 
 In fact, the possibility of manipulating $\mathfrak A(X)$s for $X\in \X$ seems to be related to
 the content of the paper \cite{ghk}, where Akemann-Doner C*-algebras were investigated.
 It seems that using the techniques of this paper one could try to investigate the 
 densities of C*-subalgebras of the algebras $\pi[\M_{\X, \mathfrak A}]$ which have
 commutative lifts to $\bb$.

 Now we would like to vary the families $\X$ which induce 
 $\M_{\X, \mathfrak A}$ in order to see that we have maximal  number (equal to $2^\cc$)  
 of pairwise non-$*$-isomorphic masas of $\QQ$ of the form $\pi[\M_{\X, \mathfrak A}]$.
 
 \begin{definition}\label{objects-a} $ $
 \begin{enumerate}
 \item  $Y_k$ for $k\in \N$ denotes the set of all elements of $\N$ divisible by the $(k+1)$-th
 prime number.
\item $h_0: Clop(\{0,1\}^\N)\rightarrow \wp(\N)$ is the Boolean monomorphism
 such that  for each $k\in\N$ we have 
 $$h_0(\{x\in \{0,1\}^\N: x(k)=1\})=Y_k.$$
 \item For $x\in \{0,1\}^\N$ by $\X_x$ we denote
 $$\{X\subseteq \N: X\ \hbox{is infinite}\  \&\  \forall n\in \N\   X\subseteq^* h_0([x|n])\}$$
 \item For $x\in \{0,1\}^\N$ by $\Y_x$ we denote a fixed infinite
 maximal almost disjoint family of infinite subsets from $\X_x$.
\item For $a\subseteq \{0,1\}^\N$ we define
$$\X_a=\bigcup\{\X_x: x\in a\}\cup\bigcup\{\Y_x: x\in \{0,1\}^\N\setminus a\}.$$
\item For $a\subseteq \{0,1\}^\N$ we fix ${\mathfrak A}_a\in (71/72, 1]^{\X_a}$  such that
$$\{{\mathfrak A}_a^{-1}[\{r\}]: r\in (71/72, 1]\}=
\{\X_x: x\in a\}\cup\{\{X\}: X\in \bigcup\{\Y_x: x\in \{0,1\}^\N\setminus a\}\}.$$

 \end{enumerate}
 \end{definition}

 \begin{lemma}\label{a-wide} Suppose that $a\subseteq \{0,1\}^\N$. 
   Then $\X_{a}$ is wide.
 \end{lemma}
 \begin{proof} Let $Y\subseteq \N$ be infinite. 
By induction on $n\in \N$ define $x=(x(n))_{n\in \N}\in \{0,1\}^\N$ such that
 $h_0([x|n])\cap Y$ is infinite for every $n\in \N$. So there is an infinite $Z\subseteq \N$ such that
 $Z\subseteq^* h_0([x|n])$ for every $n\in \N$ and $Z\subseteq Y$. 
 If $x\in a$, then $Z\in \X_x$ and we are done. If $x\not\in a$, then by the maximality of $\Y_x$ there is
 $X\in \Y_x$  such that $X\cap Z$ is infinite,
 and so $X\cap Y$ is infinite, so we are done as well. 
 \end{proof}
 
 \begin{lemma}\label{a-coherent} Suppose that $a\subseteq \{0,1\}^\N$. 
  Then ${\mathfrak A}_a$ is coherent.
 \end{lemma}
 \begin{proof}
 We check Definition \ref{def-coherent}. Let $X, Y\in \X_a$. Note that
 the defining property (6) of Definition \ref{objects-a} yields that
 ${\mathfrak A}_a(X)\not={\mathfrak A}_a(Y)$ implies that $X$ and $Y$ are almost disjoint.
 \end{proof}
 
  \begin{proposition}\label{a-prop} Suppose that $a\subseteq \{0,1\}^\N$.
 Then $\pi[\M_{\X_a, {\mathfrak A}_a}]$ is a masa of $\QQ$ generated by projections.
 \end{proposition}
 \begin{proof}
 Use Lemmas \ref{a-wide} and \ref{a-coherent} and Proposition \ref{prop-zfc}.
 \end{proof}
 
 \begin{lemma}\label{restriction-M} Suppose that $a\subseteq \{0,1\}^\N$.
 Then for every $x\in \{0,1\}^\N\setminus a$
   there is a projection of the form $P_X \in \D_0\subseteq \M_{\X_a, {\mathfrak A}_a}$ such that
   \begin{enumerate}
   \item $X\subseteq^*h_0([x|n])$ for all $n\in \N$, where $h_0$ is as in Definition \ref{objects-a}
 \item $P_X\M_{\X_a, {\mathfrak A}_a}P_X$ is $*$-isomorphic to
 the algebra $\M_{\Z, {\mathfrak A}''}$  for a maximal almost disjoint family $\Z\subseteq \wp(\N)$
 and an injective ${\mathfrak A}''\in (71/72, 1]^\Z$.
 \end{enumerate}
 In particular, if $a$ is a proper subset of $\{0,1\}^\N$, then 
  $\pi[\M_{\X_a, {\mathfrak A}_a}]$ does not have a commutative lift to $\bb$ nor does it admit a conditional expectation from
  $\QQ$.
 \end{lemma}
 \begin{proof} Let $x\in \{0,1\}^\N\setminus a$.
 Following the notation of Definition \ref{objects-a} let $\Y_x$ be the maximal almost disjoint
family in $\X_x$. Let 
 $X_k\in \Y_x$ be distinct for $k\in \N$. 
 By the $H_\omega'$ property of $\wp(\N)/Fin$  (Lemma 1.1.2 of \cite{handbook-van-mill})
 there is $X\subseteq \N$ such that $X_k\subseteq^* X\subseteq^* h_0([x|n])$ for every $n, k\in \N$.
 Let 
 $$\Y=\{Y'\cap X: Y'\in \Y_x\ \hbox{and $Y'\cap X$ is infinite}\}.$$
 As $\Y_x$ is a maximal almost disjoint family among the sets  almost below each $h_0([x|n])$
 and $X$ is not covered by the union of a finite set and finitely many elements of $\Y_x$ (as it
 is almost above all $X_k$s)
 $\Y$ is an infinite maximal  almost disjoint family of infinite subsets of $X$.  So
 if $\sigma: X\rightarrow \N$ is a bijection, then
 $\Z=\{\sigma[Y]: Y\in \Y\}$ is a maximal almost disjoint family of subsets of $\N$.
 Having in mind Definition \ref{def-M} it should be clear that 
 $$\mathcal O_\Y=
 \bigcap_{Y\in \Y}\{T\in \A_0(X): P_{Y}TP_{Y}\in \D_\kk(Y, \alpha_{Y}')\}$$
 is $*$-isomorphic (via associations induced by $\sigma$) to  the algebra $\M_{\Z, {\mathfrak A}''}$  
 for a maximal almost disjoint family $\Z\subseteq \wp(\N)$
 and an appropriate injective ${\mathfrak A}''\in (71/72, 1]^\Z$, where
 $\alpha_{Z}''=\alpha_Y'=\alpha_{Y'}$ for $Y'\in \Y_x$ and $Z\in \Z$  such that $Z=\sigma[Y'\cap X]$ 
 and $Y=Y'\cap X$ (note that the almost disjointness
 of $\Y_x$ gives us the uniqueness of such $Y'$ and so the uniqueness of $Z$ and the uniqueness of $Y$).
 
 So it is enough to show that 
 $$\mathcal O_\Y=P_X\M_{\X_a, \mathfrak A_a}P_X=
  P_X\Big(\bigcap_{Y'\in \X_a}\{T\in \A_0: P_{Y'}TP_{Y'}\in \D_\kk(Y', \alpha_{Y'})\}\Big)P_X. $$

 For the inclusion $\subseteq$ take $T\in \mathcal O_\Y$ and $Y'\in \X_a$.
 If $Y'\cap X$ is finite, then $P_{Y'}TP_{Y'}$ is finite dimensional, so compact, so 
 $P_{Y'}TP_{Y'}\in \D_\kk(Y', \alpha_{Y'})$ because $T\in \A_0(X)$.
 If $Y'\cap X$ is infinite, then by the choice of $X$
 we must have $Y'\in \Y_x$ and so $Y'\cap X=Y\in \Y$ and so $P_{Y}TP_{Y}\in \D_\kk(Y, \alpha_{Y}')$,
 hence $P_{Y'}TP_{Y'}\in \D_\kk(Y', \alpha_{Y'})$, because $T\in \A_0(X)$ and $\alpha_{Y}'=\alpha_{Y'}$.
 So $T\in \M_{\X_a, \mathfrak A_a}$ and $T=P_XTP_X$ since $T\in \A_0(X)$, so 
 $T\in P_X\M_{\X_a, \mathfrak A_a}P_X$ as needed.
 
 For the inclusion $\supseteq$ take $T\in \M_{\X_a, \mathfrak A_a}$ and 
 $Y\in\Y$. Then $P_XTP_X\in \A_0(X)$ and  $P_{Y}TP_{Y}\in \D_\kk(Y, \alpha_{Y}')$
 since $P_{Y'}TP_{Y'}\in \D_\kk(Y', \alpha_{Y'})$ for $Y'\in \Y_x$ such that $Y=Y'\cap X$
 because $\alpha_{Y}'=\alpha_{Y'}$. So $P_XTP_X\in \mathcal O_\Y$.
 
 To prove the last part of the lemma first note that by its main part and by Proposition \ref{few-projections} (4)
 $Proj(\pi[\M_{\X_a, \mathfrak A_a}])$ does not satisfy condition $H_\omega'$. So we can argue as in the last part of the proof of
  Proposition \ref{few-projections} (4) that $\pi[\M_{\X_a, \mathfrak A_a}]$  does not have a commutative lift to $\bb$.
 Secondly note that by the main part of the lemma and  Proposition \ref{few-projections} (5)
  $\pi[\M_{\X_a, \mathfrak A_a}]$ does not have the Grothendieck property
  as $\pi(P_X)\pi[\M_{\X_a, \mathfrak A_a}]\pi(P_X)$  is a quotient of $\pi[\M_{\X_a, \mathfrak A_a}]$ 
  and quotients of Banach spaces preserve the Grothendieck property by Proposition 3.1.4 of \cite{kania}.
  So we can argue as in the last part of the proof of
  Proposition \ref{few-projections} (5) that $\pi[\M_{\X_a, \mathfrak A_a}]$ does not
  admit a conditional expectation from $\QQ$.
 \end{proof}
 
  \begin{definition}
 Suppose that $a\subseteq \{0,1\}^\N$ and $\B\subseteq \QQ$ is a commutative
  C*-subalgebra of $\QQ$. We say that a  Boolean monomorphism
 $h: Clop(\{0,1\}^\N)\rightarrow \B$ captures $\B$ on the set $a$ if the following two conditions are equivalent
 for every $x\in\{0,1\}^\N$:
 \begin{enumerate}
 \item $x\in a$
 \item for any nonzero projection $P\in \B$ satisfying $P\leq h({[x|n]})$ for all $n\in \N$, the algebra
 $Proj(P\B P)$ is Boolean isomorphic to $\wp(\N)/Fin$.
 \end{enumerate}
 \end{definition}

 \begin{lemma}\label{captures} Suppose that $a\subseteq \{0,1\}^\N$ and
  $h_0: Clop(\{0,1\}^\N)\rightarrow \wp(\N)$ is the Boolean embedding as in Definition \ref{objects-a} (2)
  and let $g_0:\wp(\N)\rightarrow Proj(\QQ)$ be defined by $g_0(X)=\pi(P_X)$.
  Let $\X_a$ and ${\mathfrak A}_a$ be as in Definition \ref{objects-a} (5), (6).
  Then $$g_0\circ h_0:  Clop(\{0,1\}^\N)\rightarrow \pi[Proj(\M_{\X_a, {\mathfrak A}_a})]$$
  captures $\pi[\M_{\X_a, {\mathfrak A}_a}]$ on the set $a$.
  
 \end{lemma}
 \begin{proof}  First note that any projection in $\pi[\M_{\X_a, {\mathfrak A}_a}]$ is in
 $\pi[Proj(\M_{\X_a, {\mathfrak A}_a})]$ because
  $\M_{\X_a, {\mathfrak A}_a}$ is generated by its projections by Lemma \ref{gen-projections}, so
  $\pi[\M_{\X_a, {\mathfrak A}_a}]$ is generated by 
  $\pi[Proj(\M_{\X_a, {\mathfrak A}_a})]$. However, if a commutative C*-algebra $C(K)$ is generated by
  some Boolean algebra $\CC$ of its projections, then $\CC$ contains all projections (any
  clopen set can be written as a finite union of finite intersections of clopen sets corresponding to elements of $\CC$, because
  of the compactness and separation of points of $K$ by $\CC$).
  
 If $x\in a$, 
fix a projection $P=(P_n)_{n\in \N}\in Proj(\M_{\X_a, {\mathfrak A}_a})$ such that $\pi(P)\neq0$ and $\pi(P)\leq g_0(h_0([x|n]))$ for each $n\in \N$.
Define
$X=\{n\in \N: P_n\not=0\}$.  Note that 
$$\pi(P)\leq \pi(P_X)\leq \pi(P_{h_0([x|n])})$$
and $X\subseteq^*h_0([x|n])$ for all $n\in\N$ by Lemma \ref{limit-compact}. 
Since $\pi(P)\neq0$, the set $X$ is infinite. It follows that $X\in \X_x\subseteq \X_a$.

It is enough to prove that  $\pi[Proj(P_X\M_{\X_a, {\mathfrak A}_a} P_X)]$ is Boolean isomorphic to $\wp(\N)/Fin$
since every nonzero principal ideal of $\wp(\N)/Fin$ is Boolean
isomorphic to $\wp(\N)/Fin$, so the same conclusion holds below every
nonzero projection $\pi(P)\leq \pi(P_X)$ in $\pi[\M_{\X_a, {\mathfrak A}_a}]$.
 But
 $P_X\M_{\X_a, {\mathfrak A}_a}P_X=\D_\kk(X, \alpha_X^a)$ by Lemma \ref{Ds-inM} as $X\in \X_a$, where 
 $\alpha_X^a={\mathfrak A}_a(X)$. So
$$Proj(\pi[\D_\kk(X,\alpha_X^a)])=Proj(\pi[\D(X,\alpha_X^a)])$$
is Boolean isomorphic to $\wp(X)/Fin$, and hence to $\wp(\N)/Fin$, as needed.

 If $x\in \{0,1\}^\N\setminus a$, then by Lemma \ref{restriction-M} we have a projection
 $P_X$ as described there. By Proposition \ref{few-projections} $Proj(\pi(P_X)\pi[\M_{\X_a, {\mathfrak A}_a}] \pi(P_X))$
 does not satisfy condition $H_\omega'$ and so is not Boolean isomorphic to
 $\wp(\N)/Fin$ by Lemma \ref{h-omega}.
 \end{proof}

   It seems that further understanding of the Gelfand spaces of 
   the masas $\pi[\M_{\X_{a}, {\mathfrak A}_a}]$ may be related to
   the arguments like in \cite{ghk}. However, we are already able to prove the following:
 
\begin{proposition}\label{two-to-c} Use the notation
of Definition \ref{objects-a}. Among all the masas  $\pi[\M_{\X_{a}, {\mathfrak A}_a}]$ for
$a\subseteq \{0,1\}^\N$ there are $2^{\cc}$
pairwise non-$*$-isomorphic ones.
\end{proposition}
\begin{proof} Suppose that $a, a'\subseteq \{0,1\}^\N$ and
$\phi: \pi[\M_{\X_{a}, {\mathfrak A}_a}]\rightarrow \pi[\M_{\X_{a'}, {\mathfrak A}_{a'}}]$ 
 is a $*$-isomorphism.
 Then, as $g_0\circ h_0$ as in Lemma \ref{captures} captures  $\pi[\M_{\X_{a}, {\mathfrak A}_a}]$ 
 on the set $a\subseteq\{0,1\}^\N$ by Lemma \ref{captures}, 
 we conclude that $\phi\circ g_0\circ h_0$ captures 
 $\pi[\M_{\X_{a'}, {\mathfrak A}_{a'}}]$ on the set $a$ as well.
 As, due to  the counting argument, there are only continuum many Boolean homomorphisms 
 from the countable set $Clop(\{0,1\}^\N)$  into
 $\QQ$ which has cardinality $\cc$, a given masa can be captured by Boolean
  monomorphisms only on continuum many sets $a\subseteq \{0,1\}^\N$.
 So an algebra of the form $\pi[\M_{\X_{a'}, {\mathfrak A}_{a'}}]$  can be isomorphic only to continuum many
 other algebras of the form $\pi[\M_{\X_{a}, {\mathfrak A}_{a}}]$. But we have $2^{\cc}$ many of these
 algebras, one for each $a\subseteq \{0,1\}^\N$.
\end{proof}

\bibliographystyle{amsplain}

\end{document}